\documentclass
{amsart}

\usepackage{amssymb,amscd}
\usepackage[all,line,arc,curve,color,frame,pdf]{xy}
\usepackage{tikz} 
\usepackage{tikz-cd}
\usetikzlibrary{positioning, trees, snakes}
\usepackage[textsize=tiny]{todonotes}
\usepackage{hyperref}
\usepackage[shortlabels]{enumitem}
\usepackage{float} 
\usepackage[paper=a4paper, margin=3.5cm]{geometry}
\usepackage{cleveref}  
\usepackage{comment}
\usepackage{scalerel}
\usepackage[normalem]{ulem}
\usepackage{booktabs}
\usepackage{mathtools}
\usepackage{cite} 

\numberwithin{equation}{section}
\theoremstyle{plain}
\newtheorem{thm}{Theorem}[section]
\newtheorem{cor}[thm]{Corollary}
\newtheorem{lemma}[thm]{Lemma}
\newtheorem{prop}[thm]{Proposition}

\newtheorem{question}[thm]{Question}

\theoremstyle{remark}
\newtheorem{rem}[thm]{Remark}
\newtheorem{ex}[thm]{Example}

\theoremstyle{definition}
\newtheorem{defi}[thm]{Definition}

\def\SS{\mathfrak{S}}
\def\ctop{[X_M]}
\def\Prod{\prod}
\def\M{\mathcal{M}}

\DeclareMathOperator{\corank}{crk}\let\crk\corank

\DeclareMathOperator{\Pic}{Pic}
\DeclareMathOperator{\Perm}{\Pi}
\DeclareMathOperator{\perm}{\Pi}
\DeclareMathOperator{\CH}{CH}
\DeclareMathOperator{\Mon}{Mon}
\DeclareMathOperator{\Fl}{Flat}
\DeclareMathOperator{\FlFl}{FlFlat}
\DeclareMathOperator{\CQ}{CQ}
\DeclareMathOperator{\CCl}{CC}

\DeclareMathOperator{\nullity}{nl}

\def\Ff{\mathcal{F}}

\DeclareMathOperator{\leng}{length}

\def\bfa{\mathbf{a}}
\def\bfb{\mathbf{b}}
\def\bfc{\mathbf{c}}

\subjclass[2020]{primary: 14C17, 14N15, 52A39, 05E14, 52B40 
}
\usepackage[textsize=tiny]{todonotes}

\usepackage{enumitem}

\newcommand\ignore[1]{}

\DeclareMathOperator{\conv}{conv}

\newcommand\CC{{\mathbb{C}}}
\newcommand\PP{{\mathbb{P}}}
\newcommand\RR{{\mathbb{R}}}
\newcommand\QQ{{\mathbb{Q}}}
\newcommand\ZZ{{\mathbb{Z}}}

\def\C{{\mathbb C}}

\def\N{{\mathbb N}}

\def\Q{{\mathbb Q}}
\def\R{{\mathbb R}}
\def\Z{{\mathbb Z}}

\def\cF{{\mathcal F}}

\def\Q{{\mathbb{Q}}}

\def\operatorname#1{\mathop{\rm #1}\nolimits}

\def\Hom{\operatorname{Hom}}

\def\Pic{\operatorname{Pic}}
\def\Hom{\operatorname{Hom}}

\def\rank{\operatorname{rk}}
\def\rk{\operatorname{rk}}

\newcommand{\pb}{\ar@{}[dr]|{\text{\pigpenfont J}}}

\newcommand{\xdasharrow}[2][->]{
\tikz[baseline=-\the\dimexpr\fontdimen22\textfont2\relax]{
\node[anchor=south,font=\scriptsize, inner ysep=1.5pt,outer xsep=2.2pt](x){#2};
\draw[shorten <=3.4pt,shorten >=3.4pt,dashed,#1](x.south west)--(x.south east);
}}

\begin{document}
\title{Mixed Eulerian numbers and beyond}


\author[Liu]{Gaku Liu}
\address{University of Washington, Seattle, USA}
\email{gakuliu@uw.edu}
\author[Micha{\l}ek]{Mateusz Micha{\l}ek}
\address{
	University of Konstanz, Germany, Fachbereich Mathematik und Statistik, Fach D 197
	D-78457 Konstanz, Germany
}
\email{mateusz.michalek@uni-konstanz.de}
\author[Weigert]{Julian Weigert}
\address{Max Planck Institute for Mathematics in the Sciences/University of Leipzig, Leipzig, Germany
}
\email{julian.weigert@mis.mpg.de}

\begin{abstract}
 We derive explicit formulas for the matroidal mixed Eulerian numbers. We resolve a question posed by Berget, Spink, and Tseng, demonstrating that the invariant defined by matroidal mixed Eulerian numbers is precisely equivalent to Derksen's $\mathcal{G}$-invariant. As an application, we provide the first explicit, non-recursive formula for mixed Eulerian numbers. Our combinatorial approach draws inspiration from the classical work of Schubert and incorporates the cutting-edge contributions of Huh.
\end{abstract}

\thanks{GL is supported by the NSF grant DMS-2348785}
\thanks{MM is supported by the DFG grant 467575307}

\maketitle

\section{Introduction}
\subsection{Mixed Eulerian numbers}
Eulerian numbers, first introduced by Euler in the mid-18th century, have since played a central role in mathematics. The Eulerian number $A(n+1,k)$ is the number of permutations of $[n]:=\{0,\dots,n\}$ with exactly $k$ ascents, i.e., elements that are larger than the preceding element in the permutation.

Consider the $n$-dimensional $k$-th hypersimplex, denoted $\Delta_{n,k}$. It is defined as the intersection of the $(n+1)$-dimensional unit cube with the hyperplane $\sum_{i=0}^{n} x_i=k$. From a geometric perspective, $A(n,k)$ is the normalized volume of $\Delta_{n,k}$. 
This geometric viewpoint enables the interpretation of any Eulerian number $A(n,k)$ as the degree of the projective toric variety $X_{\Delta_{n,k}}$ associated with the hypersimplex $\Delta_{n,k}$. Equivalently, over the complex numbers $\CC$, it represents the number of intersection points of $X_{\Delta_{n,k}}\subset \PP\left(\CC^{\binom{n+1}{k}}\right)$ with $n$ general hyperplanes, see e.g.~\cite[Chapter 2 and 8]{jaBernd}, \cite{CoxLittleSchenck}.

While many results, including explicit closed formulas for Eulerian numbers, are well-known, their generalizations—the \emph{mixed Eulerian numbers} $A(a_1,\dots,a_n)$ for nonnegative integers $a_i$ summing to $n$—are much less understood. These were introduced by Postnikov in his seminal work \cite{postnikov2009permutohedra} as the \emph{mixed volume} of $a_1$ (hyper)simplices $\Delta_{n,1}$, $a_2$ hypersimplices $\Delta_{n,2}$, $\dots$, and $a_n$ (hyper)simplices $\Delta_{n,n}$. Mixed Eulerian numbers generalize classical Eulerian numbers, Catalan numbers, binomial coefficients, factorials, as well as "many other combinatorial sequences" \cite[p.~1028]{postnikov2009permutohedra}.

A small clash of notation exists: the Eulerian number $A(n,k)$ is equal to the mixed Eulerian number $A(0,\dots,0,n,0,\dots,0)$, with $n$ in the $k$-th position. For combinatorial interpretations of mixed Eulerian numbers and various recursive formulas, we refer to \cite{ehrenborg1998mixed}, \cite{liu2015mixed}, \cite{postnikov2009permutohedra}, and \cite{croitoru2010mixed}. However, until now, no closed formula was known. In fact, Croitoru writes that "finding a simple closed formula for $A(a_1,\dots,a_n)$ is unlikely (there is no such formula already for $A(0,\dots,0,k,n-k,0,\dots,0)$)" \cite[p.~18]{croitoru2010mixed}. 

One of our main results, Corollary \ref{cor:mixedEuformula}, provides a closed formula for all mixed Eulerian numbers $A(a_1,\dots,a_n)$. The main idea, with geometric inspirations going back to Schubert, is to interpret $A(a_1,\dots,a_n)$ as coefficients (up to multinomial factors) 
of a homogeneous polynomial $P$ in $n$ variables. For experts, this is the volume polynomial for hypersimplex classes in the Chow ring of the permutohedral variety---more details are presented below. Performing an easy linear change of coordinates on $P$, we obtain a new polynomial, whose coefficients have now a straightforward interpretation in terms of binomial coefficents. 

\subsection{Geometry of matroidal mixed Eulerian numbers}
As we explain next, our approach naturally extends to the so-called \emph{matroidal mixed Eulerian numbers} studied in \cite{katz2023matroidal} and \cite{berget2023log}. A permutohedron $\Perm_n$ is an $n$-dimensional lattice polytope that is the convex hull of all vectors in $\RR^{n+1}$ with coordinates given by the permutations of the set $[n]$. Equivalently, $\Perm_n=\sum_{i=1}^{n}\Delta_{n,i}$, i.e.~it is the Minkowski sum of all hypersimplices of dimension $n$.  The permutohedral variety $X_{\Perm_n}$ is a smooth toric variety associated to the permutohedron $\Perm_n$. Let $M$ be a loopless matroid on the set $[n]$ of rank $\rk M$. 
By the groundbreaking work \cite{June1}, \cite{June2} and \cite{adiprasito2018hodge} one naturally associates to $M$ a cohomology class $[X_M]$ in the degree $(n-\rk M+1)$ part of the Chow ring $\CH (X_{\Perm_n})$ of the permutohedral variety $X_{\Perm_n}$. The presentation of $\Perm_n$ as the Minkowski sum of hypersimlices gives the embedding $X_{\Perm_n}\subset \prod_{i=1}^{n} \PP\left(\CC^{\binom{n+1}{i}}\right)$. Consequently, it is natural to introduce $n$ base-point-free divisors $L_1,\dots,L_{n}$ on $X_{\Perm_n}$, where $L_i\in \CH^1(X_{\Perm_n})$ is the pull-back of the hyperplane section in $\PP\left(\CC^{\binom{n+1}{i}}\right)$. The mixed matroidal Eulerian numbers are defined as 
\[A_M( a_1,\dots,a_{n}):=\int_{X_{\Perm_n}} [X_M]L_1^{a_1}\cdots L_{n}^{a_{n}}\in\ZZ,\]
where $\int_{X}$ is the standard notation of taking the degree and we assume that $a_i$'s are nonnegative integers summing to $(\rk M-1)$. The mixed Eulerian numbers correspond to the case where $M$ is the uniform matroid $U_{n+1,n+1}$ of rank $n+1$ and $[X_M]=[X_{\Perm_n}]$. Thus $A(a_1,\dots,a_n)=\int_{X_{\Perm_n}}L_1^{a_1}\cdots L_n^{a_n}$.
In particular, the aforementioned polynomial $P$ is simply \[P(x_1,\dots,x_n):=\int_{X_{\Perm_n}} \left(\sum_{i=1}^n x_i L_i\right)^n\in \ZZ[x_1,\dots,x_n].\] This is a prototypical example of a volume polynomial and a Lorentzian polynomial \cite{branden2020lorentzian}.

The divisors $L_i$ do not linearly span the rational Picard group $\Pic(X_{\Perm_n},\QQ)=\CH^1(X_{\Perm_n},\QQ)$. Indeed, there is a natural action of the symmetric group $\SS_{n+1}$ on $\Perm_n$, which induces an action on the Chow ring $\CH(X_{\Perm_n},\QQ)$. The divisors $L_i$ are $\SS_{n+1}$ invariant. We prove in Lemma \ref{invariant part} that indeed $L_i$'s generate the invariant ring $\CH(X_{\Perm_n},\QQ)^{\SS_{n+1}}$ and therefore, linearly span the $\SS_{n+1}$ invariant part of the rational Picard group. In fact the divisors $L_1,\ldots,L_n$ are precisely the rays of the cone of numerically effective (nef) $\SS_{n+1}$-invariant divisors, see Lemma \ref{lem:LDivisorsspanNefCone}. To sum up we show the following.\begin{prop}
    \begin{enumerate}[label=(\alph*)]
        \item The divisors $L_1,\ldots,L_n$ span the $\SS_{n+1}$-invariant ring \\ $\CH(X_{\Perm_n},\QQ)^{\SS_{n+1}}$ as a $\QQ$-subalgebra of $\CH(X_{\Perm_n},\QQ)$
        \item The cone spanned by $L_1,\ldots,L_n$ inside the rational Picard group $\Pic(X_{\Perm_n},\QQ)$ is the intersection of the nef cone of $X_{\Perm_n}$ with the $\SS_{n+1}$ invariant part of the rational Picard group.
    \end{enumerate}
\end{prop}

An equivalent construction of the permutohedron realizes it as a sequence of blow-ups:
\[X_{\Perm_n}=X_n\simeq X_{n-1}\rightarrow X_{n-2}\rightarrow\cdots\rightarrow X_1\rightarrow X_0\rightarrow \PP(\CC^{n+1}),\]
where for $i=0,\dots,n-1$ the map $X_{i+1}\rightarrow X_i$ is the blow-up of the strict transform of the union of $i$-dimensional
coordinate subspaces of $\PP(\CC^{n+1})$. For $i=1,\dots,n$ we define $S_i$ as the $i$-th exceptional divisor. We note that this divisor has several components, that is why $S_i$'s also do not linearly span $\Pic_{\QQ}(X_{\Perm_n})$. Indeed, their linear span over $\QQ$ (but not over $\ZZ$) is the same as that of $L_i$'s. The linear relations among the divisors for $i=1,\dots,n$ are:
\begin{align}\label{eqn:lineartransLiSi}
S_i=& -L_{i-1}+2L_i-L_{i+1}\\
\notag L_i=& \sum_{j=1}^n \frac{\min(i,j)\left(n+1-\max(i,j)\right)}{n+1} S_j,
\end{align}
where we formally set $L_0=L_{n+1}=0$. Motivated by these equations we define the following matrices.
\begin{defi}\label{def:AB}
Let $A_{n,L\rightarrow S}$ be an $n\times n$ matrix defined as:
\[(A_{n,L\rightarrow S})_{i,j}=\begin{cases}
2 & \text{ if }i=j\\
-1 & \text{ if }i=j\pm 1\\
0 & \text{otherwise}. 
\end{cases}
\]
Let $B_{n,S\rightarrow L}$ be an $n\times n$ matrix defined as:
\[(B_{n,S\rightarrow L})_{i,j}=\frac{\min(i,j)\left(n+1-\max(i,j)\right)}{n+1}.
\]
We note that $A_{n,L\rightarrow S}$ and $B_{n,S\rightarrow L}$ are mutually inverse to one another.
\end{defi}

For any matrix $A$ we define $S^d A$ as the $d$-th symmetric power of $A$, i.e.~the transformation induced by $A$ on homogeneous degree $n$ polynomials. 
\begin{defi}\label{def:coefsC}
We denote by $\Mon^d(x_1,\dots,x_n)$ the set of degree $d$ monomials in variables $x_1,\dots,x_n$. We will have two shorthand notations for monomials. Given an $n$-tuple $\bfa = (a_1,\dots,a_n)$ of nonnegative integers, we define $x^\bfa := \prod_{i=1}^n x_i^{a_i}$. In addition, given $k$-tuples $\bfb = (b_1,\dots,b_k)$ and $\bfc = (c_1,\dots,c_k)$ of \emph{positive} integers with $1 \le b_1 < \dots < b_k \le n$, we define $x_\bfb^\bfc := \prod_{i=1}^k x_{b_i}^{c_i}$. Thus, every monomial in $x_1,\dots,x_n$ can be represented uniquely as $x^\bfa$ and uniquely as $x_\bfb^\bfc$.
\end{defi}

It is natural to label the columns of $S^dA_{n,L\rightarrow S}$ by $\Mon^d(L_1,\dots,L_n)$ and the rows by $\Mon^d(S_1,\dots,S_n)$, and reciprocally for $S^dB_{n,S\rightarrow L}$. The entry of $S^dB_{n,S\rightarrow L}$ labeled by row $L^\bfa$ and column $S_\bfb^\bfc$ is denoted $C_\bfb^\bfc(\bfa)$. In other words, for $L^\bfa \in \Mon^d(L_1,\dots,L_n)$ we have
\begin{align}
\label{LfromS}
    L^\bfa = \sum_{S_\bfb^\bfc \in \Mon^d(S_1,\dots,S_n)} C_\bfb^\bfc(\bfa)S_\bfb^\bfc.
\end{align}

The linear transformations above are surprisingly similar to the formulas used by Schubert for the divisors on the variety of complete quadrics over 150 years ago. Schubert's magnificent idea is to translate an interesting geometric invariant, expressed as intersection of $L_i$'s, into less meaningful, but easier to compute sequences of invariants that involves intersection products of $S_i$'s (and $L_i$'s). In fact the cohomology class of the locus of quadrics tangent to a general $i$-dimensional subspace is the analog of the divisor $L_i$ on the variety of complete quadrics, for more details see Section \ref{sec:permCC} and \cite{dinu2021applications}. Throughout the years the computation of intersection numbers involving $L_1,\ldots,L_n$, known as characteristic numbers, was pushing forward the foundational work in intersection theory and vice-versa theoretical advancements allowed to obtain new numerical information, see e.g.~\cite[Chapter 8.3]{eisenbud20163264}, \cite{ThaddeusComplete, ja2, massarenti2020birational, LaksovLascouxThorup, Pragacz, DeConciniProcesi1} and references therein. 

As we explain in Section \ref{sec:permCC} the fact that we obtain the same formulas is by far not incidental. The permutohedral variety $\Perm_n$ is included in the variety of complete quadrics, which further is a subvariety of the variety of complete collineations \cite{dinu2021applications, michalek2023enumerative, ja1, occhetta2023chow}. Schubert's divisors $L_{i,\CQ_n}$ (resp. $L_{i,\CCl_n}$) and $S_{i,\CQ_n}$ (resp. $S_{i,\CCl_n}$) linearly span the rational Picard group of the latter two varieties. Pull-back via inclusion to the permuthedral variety induces an injective, but not surjective, linear map on rational Picard groups, with the image given by  $\Pic(X_{\Perm_n},\QQ)^{\SS_{n+1}}$. This linear map simply sends Schubert's divisors $L_{i,\CQ_n}$ (resp. $L_{i,\CCl_n}$) and $S_{i,\CQ_n}$ (resp. $S_{i,\CCl_n}$) to the divisors $L_i$ and $S_i$ on $X_{\Perm_n}$ as defined above. 
This allows us to apply many ideas from classical intersection theory to the combinatorial setting. A posteriori, every single argument may be translated in combinatorial terms. We would like to emphasize that the geometric intuition behind the divisors $L_i$ and $S_i$ was indispensable for our reasoning. Still, apart from Section \ref{sec:permCC}, we decided to keep relations to geometry only in remarks, which may be skipped by readers focusing only on combinatorics. Geometrically minded readers may find some of our formulas technical and we hope this group will enjoy the geometric intuition, as we did. 

One of the most exceptional results applying intersection theory to combinatorics was obtained by Huh and later in a joint work with Adiprasito and Katz \cite{June1, adiprasito2018hodge}. It confirmed log-concavity of the absolute values of the coefficients of the characteristic polynomial of a matroid. 
\begin{defi}
Let $M$ be a matroid on a groundset $E$.  For $S\subseteq E$ we write $\nullity_M(S):=|S|-\rank_M(S)$ for its \emph{nullity}. We define the \emph{reduced characteristic polynomial} $\chi_M(q)\in \Z[q]$: \begin{align*}
    \chi(q):=(-1)^{\rank(M)}\frac{T_M(1-q,0)}{q-1}
\end{align*}
where $T_M(x,y)=\sum_{S\subseteq E}(x-1)^{\rank(M)-\rank_M(S)}(y-1)^{\nullity_M(S)}\in \Z[x,y]$ is the Tutte polynomial of $M$. 
\end{defi}
In \cite{June2} the authors prove that for $l=0,\ldots,\rank(M)-1$ the matroidal mixed Eulerian number $A_M(l,0,\ldots,0,\rank(M)-1-l)$ equals the absolute value of the coefficient of $q^l$ in $\chi_M(q)$. We will denote this number by $\gamma_M(l)\in \ZZ$. These numbers can be understood explicitly from the structure of the lattice of flats of $M$, see for example \cite[Proposition 2.4]{June2}.

\subsection{Statements of the main results}
We start with our result about classical mixed Eulerian numbers.
\begin{thm}\label{thm:mixedEu}
The following are true.
\begin{enumerate}
\item For $S_\bfb^\bfc \in \Mon^n(S_1,\dots,S_n)$ where $\bfb = (b_1,\dots,b_k)$ and $\bfc = (c_1,\dots,c_k)$, we have
\[
\int_{X_{\Perm_{n}}} S_\bfb^\bfc =(-1)^{n-k} \binom{n+1}{b_1, b_2-b_1, \dots, b_k-b_{k-1}, n+1-b_k} \prod_{i=1}^{k-1} \binom{b_{i+1}-b_i-1}{\sum_{l=1}^ic_l-b_i}.
\]
\item Let $Q(x_1,\dots,x_n):=\int_{X_{\Perm_n}} \left(\sum_{i=1}^n x_i S_i\right)^n\in \ZZ[x_1,\dots,x_n]$. Then
\[
P(x_1,\dots,x_n)=\left(S^n B_{n,S\rightarrow L}\right)Q(x_1,\dots,x_n).
\]
\end{enumerate}
\end{thm}

We note that point (1) in Theorem \ref{thm:mixedEu} provides a closed formula for the coefficients of the polynomial $Q(x_1,\dots,x_n)$, which are equal to $\binom{n}{c_1,\dots,c_k}\int_{X_{\Perm_{n}}} S_\bfb^\bfc$. Point (2) is straightforward from relations \eqref{eqn:lineartransLiSi} and Definition \ref{def:AB}. Thus we see that $P(x_1,\dots,x_n)$ is a simple linear transformation of $Q(x_1,\dots,x_n)$. Finally, the mixed Eulerian number $A(a_1,\dots,a_n)$ equals the coefficient of $L^\bfa$ in $P(x_1,\dots,x_n)$ divided by $\binom{n}{a_1,\dots,a_n}$. We thus obtain the following corollary.

\begin{cor}\label{cor:mixedEuformula}
 Let $\bfa \in \N_0^n$ satisfy $\sum_{i=1}^na_i=n$. Then the associated mixed Eulerian number $A(a_1,\ldots,a_n)$ equals  
 	\[
       \smashoperator[r]{\sum_{x_\bfb^\bfc \in \Mon^n(x_1,\dots,x_n)}} C_{\bfb}^{\bfc}(\bfa)(-1)^{n-k}\binom{n+1}{b_1, b_2-b_1, \dots, b_k-b_{k-1}, n+1-b_k} \prod_{i=1}^{k-1} \binom{b_{i+1}-b_i-1}{\sum_{l=1}^ic_l-b_i},
    \]
    where $k = \leng(\bfb)$.
\end{cor}


Theorem \ref{thm:mixedEu} is derived from a more general theorem about matroidal mixed Eulerian numbers that we will present next.
\begin{defi}
For a matroid $M$ on $E$ and $S\subset E$ we denote by $M|{S}$ \emph{restriction} to $S$ and by $M/S$ \emph{contraction} of $S$. For a subset $S\subset E$ we define the \emph{corank} by $\corank_M(S):=\rank(E)-\rank(S)$.


By $\FlFl(k,M)$ we denote the set of flags of nonempty, proper flats of $M$ of length $k$. More precisely for $\Ff\in \FlFl(k,M)$ when $M$ is loopless we write $\Ff=(F_0,\dots, F_{k+1})$ where $\emptyset=F_0 \subsetneq F_1 \subsetneq \cdots \subsetneq F_k \subsetneq F_{k+1}=E$ and each $F_i$ is a flat of $M$. We define $\FlFl(k,M)$ to be empty if $\emptyset$ is not a flat of $M$, i.e. when $M$ has a loop.

\end{defi}
\begin{thm}\label{explicit matroid formula}
Let $M$ be a rank $r+1$ matroid on $n+1$ elements and let $\bfa \in \N_0^n$ satisfy $\sum_{i=1}^n a_i = r$. Then the associated matroidal mixed Eulerian number $A_M(a_1,\ldots,a_n)$ equals
	\[
       \smashoperator[r]{\sum_{S_\bfb^\bfc \in \Mon^{r}(S_1,\dots,S_n)}} C_{\bfb}^{\bfc}(\bfa) (-1)^{r-k} \smashoperator[l]{\sum_{\substack{\Ff\in \FlFl(k,M) \\ \forall i=1,\ldots,k : \\ |F_i|=n+1-b_{k+1-i} 
       }}} \prod_{i=1}^{k} \gamma_{(M | F_{i+1})/F_{i}}\left( \rk_M(F_{i+1}) - 1 -\sum_{l=1}^i c_{k+1-l} \right).  
    \]
where $k = \leng(\bfb)$ and for any matroid $N$, $\gamma_N(l) = 0$ whenever $l < 0$ or $l > \rk(N)-1$.
\end{thm}
When computing matroidal mixed Eulerian numbers in practice, the explicit formula above quickly becomes infeasible due to the large number of summands. We therefore propose several recursive formulas for these numbers. The following recursion may be seen as a generalization of Croitoru's recursive formula for mixed Eulerian numbers \cite[Theorem 2.4.6]{croitoru2010mixed}. In particular in the case of a boolean matroid $M=U_{n+1,n+1}$ our methods offer a more geometric alternative to Croitoru's combinatorial proof.
\begin{thm}
    Let $M$ be a matroid on the groundset $E=\{0,\ldots,n\}$, let $a_0,\ldots,a_n\in \N_0$ satisfy $\sum_{i=1}^na_i=\rank(M)-2$ and let $j\in \{1,\ldots,n\}$. The matroidal mixed Eulerian numbers satisfy the following recursive formula. \begin{align*}
        &A_M(a_1,\ldots,a_{j-1},a_j+1,a_{j+1},\ldots,a_n)\\&= \sum_{F\in Z(a_1,\ldots,a_n)}(B_{n,S\to L})_{j,n+1-|F|}A_{M/F}(a_1,\ldots,a_{n-|F|})A_{M|_F}(a_{n+2-|F|},\ldots,a_n)
    \end{align*}
    where the sum ranges over all elements of the set \begin{align*}
        Z(a_1,\ldots,a_n)=\left\{F \in \Fl(M)\setminus \{\emptyset,E\}\middle| a_{n+1-|F|}=0, \sum_{i=1}^{n-|F|}a_i=\corank(F)-1\right\}.
    \end{align*}
\end{thm}
Notice that we can choose the index $j$ among all indices $i$ with $a_i>0$ when computing $A_M(a_1,\ldots,a_n)$. Numerical experiments suggest that choosing $j$ as central as possible in the sequence $a_1,\ldots,a_n$ is often beneficial. We also note that the entries of the matrix $B_{n,S\to L}$ are positive, so from this recursion one can for example immediately see that all matroidal mixed Eulerian numbers are nonnegative. 

We obtain a different recursion by utilizing once more the interaction among methods working for complete collineations and the permutohedron. We can compute the $L$-divisor terms recursively and compute the $S$-divisor terms by restriction. This algorithm is motivated by \cite[Proposition 4.4, (3)]{LaksovLascouxThorup}, where a similar approach is used on varieties of complete collineations. We have the following theorem.
\begin{thm}
    Let $M$ be a matroid of rank $r+1$ on $n+1$ elements. Let $\bfa \in \N_0^n$ satisfy $\sum_{i=1}^n a_i=r$ and let $j=\max(i\mid a_i>0)$. We have $A_M(\bfa)=1$ if $j=1$ and \begin{align*}
        A_M(\bfa)&=-A_M(a_1,\ldots,a_{j-3},a_{j-2}+1,a_{j-1},a_{j}-1,0,\cdots,0)\\&+ 2A_M(a_1,\ldots,a_{j-2},a_{j-1}+1,a_{j}-1,0,\cdots,0)\\&-\sum_{\substack{F \in \Fl(M)\\|F|=n-j+2\\ \rank_M(F)=\sum_{i=j-1}^{n}a_i\\ a_{j-1}=0}}A_{M/F}(a_1,\ldots,a_{j-2})
    \end{align*}
         if $j >1$. The first summand is understood as zero if $j=2$ and the last summand appears only when $a_{j-1}=0$.
\end{thm}

In \cite{berget2023log} the authors express the evaluation $T_M(1,y)$ of the Tutte polynomial as a combination of matroidal mixed Eulerian numbers. They leave as an open problem the following question \cite[Question 1.4]{berget2023log} which we restate using the notation introduced above.
\begin{question}\label{question:mmE}
Which matroid invariants arise as the matroidal mixed Eulerian numbers?
\end{question}
The question has a very beautiful answer, by referring to Derksen's $\mathcal{G}$-invariant of a matroid \cite{derksen2009symmetric}, which is a universal valuative invariant \cite[Theorem 1.4]{derksen2010valuative}. 
\begin{thm}\label{thm:Ginv}
  Let $M_1,M_2$ be two loopless matroids of rank $r+1$ on $n+1$ elements. Then $\mathcal{G}(M_1)=\mathcal{G}(M_2)$ if and only if for every sequence $a_1,\ldots,a_n\in \Z_{\geq 0}$ with $\sum_{i=1}^n a_i=r$ we have \begin{align*}
        A_{M_1}(a_1,\ldots,a_n)=A_{M_2}(a_1,\ldots,a_n).
    \end{align*}
\end{thm}
This theorem implies that the Tutte polynomial is expressable in terms of the matroidal mixed Eulerian numbers. We provide such a formula in Proposition \ref{prop:Tutte-formula}.

\section*{Acknowledgements}
We would like to thank Andrew Berget, Alex Fink, June Huh, Eric Katz and Matt Larson for important remarks and interesting discussions. Mateusz Michalek thanks the Institute for Advanced Study for a great working environment and support through the Charles Simonyi Endowment. GL is supported by the NSF grant DMS-2348785. MM is  suported by the DFG grant 467575307.
\section{Geometry and Combinatorics}
\subsection{Matroids}
Let $E$ be a finite set. A \emph{matroid} with ground set $E$ is determined by a function $\rk_M :  2^E \to \ZZ_{\ge 0}$ satisfying the following properties:
\begin{enumerate}
\item $\rk_M(S) \le |S|$ for all $S \subset E$.
\item $\rk_M(S \cup T) + \rk_M(S \cap T) \le \rk_M(S) + \rk_M(T)$ for all $S$, $T \subset E$.
\item $\rk_M(S) \le \rk_M(S \cup \{x\}) \le \rk_M(S) + 1$ for all $S \subset E$ and $x \in E$.	
\end{enumerate}
The number $\rk_M(S)$ is the \emph{rank} of $S$ in $M$. We also define the \emph{corank} $\crk_M(S) = \rk_M(E) - \rk_M(S)$. If there is no risk of confusion, we write $\rk$ and $\crk$ instead of $\rk_M$ and $\crk_M$.

A \emph{flat} of $M$ is a set $F \subset E$ such that $\rk(F) < \rk(F \cup \{x\})$ for all $x \in E \setminus F$. We denote the set of all flats of $M$ by $\Fl(M)$. This is a graded lattice partially ordered by inclusion and graded by rank. For $M$ loopless, we denote by $\FlFl(k,M)$ the set of all flags $\cF = (F_0,F_1,\dots,F_{k+1})$ where $\emptyset = F_0 \subsetneq F_1 \subsetneq F_2 \subsetneq \dots \subsetneq F_k \subsetneq F_{k+1} = E$.


In most cases we will assume $E=[n]$. Then, we write $\Perm_E:=\Perm_n$. However, some of our arguments are inductive and we consider subsets of $E$ with induced matroid structures via contractions or deletions. Explicitly, for a subsets $S\subset E$ there is a natural structure of a matroid on the set $S$ and on $E\setminus S$ given by restriction to $S$ and contraction of $S$. To keep this distinction in mind we will also consider polytopes $\Perm_{S}$ and $\Perm_{E\setminus S}$. Abstractly they are isomorphic to  $\Perm_{|S|-1}$ and $\Perm_{|E\setminus S|-1}$. The notation allows us to naturaly identify $\Perm_{E\setminus S}\times \Perm_{|E\setminus S|-1}$ with a face of $\Perm_E$.

\subsection{The Chow ring of the permutohedral variety}\label{subsec:ChowofPer}
The description of the Chow ring of the permutohedral variety $X_{\Perm_n}$ is well-known. It follows from general theorems in toric geometry \cite{CoxLittleSchenck}. In this section we gather facts that we will need referring for more information to \cite{adiprasito2018hodge}, \cite[Chapter 4]{JuneRev}. 

Let $M\simeq \ZZ^{[n]}\simeq \ZZ^{n+1}$ be a lattice with distinguished basis $t_i$ for $i\in [n]$. The associated real vector space $M_\RR:=M\otimes_\ZZ \RR$ contains $\Perm_n$ as a codimension one lattice polytope, contained in a hyperplane where the sum of all the coordinates is constant. Let $N=M^*$ be the dual lattice and $\hat N:=N/\langle (1,\dots,1)\rangle $. Note that the lattice $N$ has a distinguished basis $e_i:=(t_i)^*$, labeled by elements $i\in [n]$. 
For $v\in N$ we write $\overline{v}\in \hat N$ for the corresponding class. For any subset $F\subset [n]$ we write $e_F:=\sum_{i\in F} e_i\in N$. 
The normal fan $\Sigma_n$ of $\Perm_n$ is naturally contained in $\hat N_\RR$. The ray generators are precisely $\overline{e_F}$ for $\emptyset\subsetneq F\subsetneq [n]$. The vectors $\overline{e_{F_1}},\dots, \overline{e_{F_k}}$ form a cone in $\Sigma_n$ if and only if $F_i$'s are pairwise comparable with respect to inclusion. In other words, up to reordering we must have $F_1\subset\dots\subset F_k$. We note that the dual lattice $(\hat N)^*$ is naturally identified with  sublattice $\hat M\subset M$ consisting of points with sum of coordinates equal to zero. Note that one may shift the permutohedron $\perm_n$, subtracting for example the lattice point $\frac{n(n+1)}{2}t_n$ to regard it as a lattice polytope inside $\hat M$.
\begin{prop}\label{prop:ChowOFperm}
\begin{enumerate}
\item The Chow ring $\CH(X_{\Perm_n},\QQ)$ equals $\QQ[x_F]_{\emptyset\subsetneq F\subsetneq [n]}/(I_l+I_q)$, where $I_l$ and $I_q$ are two ideals. The ideal $I_q$ is generated by $x_{F_1}x_{F_2}$ for $F_1, F_2$ inclusion-incomparable. The ideal $I_l$ is generated by linear relations. For any $a,b\in [n]$ the associated generator is $\sum_{F\ni a} x_F-\sum_{F\ni b} x_F\in I_l$. 
\item Every $k$ dimensional cone $\sigma\in\Sigma_n$, corresponding to $F_1\subset\dots\subset F_k$, defines a codimension $k$ subvariety of $X_{\perm_n}$. The class $[\sigma]$ of this variety is $\prod_{i=1}^k x_{F_i}\in  \CH^k(X_{\Perm_n},\QQ)$. These classes linearly span the Chow ring (but are not a basis). 
\item Every maximal cone of $\Sigma_n$ defines the same class in $\CH(X_{\Perm_n},\QQ)$, which is the class of a point. We obtain a natural isomorphism  $\CH^n(X_{\Perm_n},\QQ)\simeq \QQ$, where the class of a point is mapped to one via the degree map $\int_{X_{\Perm_n}}$.
\item For any $0\leq i\leq n$ the multiplication in the Chow ring induces the perfect pairing:
\[\CH^i(X_{\Perm_n},\QQ)\times \CH^{n-i}(X_{\Perm_n},\QQ)\rightarrow \CH^n(X_{\Perm_n},\QQ)\simeq \QQ.\]
In particular, we may define elements of $\CH^i(X_{\Perm_n},\QQ)$ via linear functions on $\CH^{n-i}(X_{\Perm_n},\QQ)$.
\end{enumerate}
\end{prop} 

Next we define two very important classes of divisors on the permutohedral variety.
\begin{defi}\label{def:Li}
For $1\leq i\leq n$ we define $L_i\in \CH^1(X_{\Pi_n},\QQ)$ by $L_i=\sum_{\emptyset\subsetneq S\subsetneq [n]} c_{i,S}x_S$ where \begin{align*}
        c_{i,S}= \begin{cases}
            \min(n+1-i,|S|) & \text{ if } n \notin S \\
            \min(n+1-i,|S|)-(n+1-i) &\text{ else}
        \end{cases}.
    \end{align*}
    The special role of $n$ in the case distinction is customary and can be replaced by any fixed $j\in [n]$ without changing the class of $L_i$ in $\CH^1(X_{\Perm_n},\QQ)$. 
\end{defi}
\begin{ex}
\label{exampleL1Ln}
    Let us look at the two divisors $L_1,L_n$. For $L_1$ we have for any $\emptyset \subsetneq S \subsetneq [n]$ that $\min(n+1-1,|S|)=|S|$ and hence \begin{align*}
        c_{1,S}= \begin{cases}
            |S| & \text{ if } n\notin S \\
            |S|-n & \text{ else}
        \end{cases}.
    \end{align*}
    The element $u=\sum_{i=0}^{n-1}(t_i-t_n)\in M$ pairs with the ray $\overline{e_S}$ to \begin{align*}
        \langle \overline{e_S},u\rangle = \begin{cases}
            |S|  & \text{ if } n\notin S \\
            |S|-1-n & \text{ else}
        \end{cases}.
    \end{align*}
    This means that $\sum_{S}\langle \overline{e_S},u\rangle x_S=0$ modulo $I_l$. Hence we can subtract this relation from the above representation of $L_1$ to cancel some of the coefficients without changing the class in $\CH(X_{\Pi_n},\QQ)$. We end up with  \begin{align*}
        L_1= \sum_{\substack{\emptyset \subsetneq S \subsetneq [n]\\ n \in S}} x_S.
    \end{align*}
    The role of $n$ as always is customary here and we can write the same equality with any fixed index $j\in [n]$. \par
    Similarly for $L_n$ we immediately get $\min(n+1-n,|S|)=1$ for all proper non-empty subsets $S$ and hence we compute \begin{align*}
        L_n=\sum_{\substack{\emptyset \subsetneq S \subsetneq [n]\\ n \notin S}} x_S.
    \end{align*}\par
\end{ex}
\begin{rem}
The divisors $L_i$ are exactly the pull-backs of the hyperplane classes via the inclusion $X_{\Perm_n}\subset \prod_{i=1}^{n} \PP^{\binom{n+1}{i}-1}$, as presented in the introduction. 
They appeared in many articles about relations of combinatorics to the Chow ring of the permutohedron. For example $L_1, L_n$ are the $\alpha, \beta$ classes in \cite{JuneRev}. The class $L_i$ is an analog of the class $\gamma_i$ from \cite{berget2023log, katz2023matroidal}. Via the correspondence of nef divisors and polytopes, the classes $L_i$ correspond to hypersimplices. 
\end{rem}
\begin{defi}\label{def:S_i}
For $i=1,\ldots,n$ we define the divisor $S_i\in \CH^1(X_{\Perm_n},\Q)$ by:
\begin{align*}
        S_i:= \sum_{\substack{S\subsetneq [n]\\ |S|=n+1-i}}x_S\in\CH^1(X_{\Perm_n},\Q).
    \end{align*}
\end{defi}
The relations \eqref{eqn:lineartransLiSi} are now straightforward to check, via Proposition \ref{prop:ChowOFperm}. 
\begin{rem}
The divisor $S_i$ is the exceptional divisor obtained in the $i$-th step of the blowup construction of $X_{\perm_n}$.
\end{rem}


The set of flats of a matroid $M$ is partially ordered by inclusion. In this way we obtain the poset of flats of $M$ denoted by $\Fl(M)$. This poset plays a central role in combinatorics of matroids and leads to the definition of the matroid class $[X_M]\in \CH^{n-\rk M+1}(X_{\perm_n},\QQ)$. This class is zero, when $M$ contains loops.

\begin{defi}
Let $M$ be a loopless matroid on the groundset $[n]$. Each maximal chain $\Ff$ in $\Fl(M)$ is naturally  identified with $\Ff\in \FlFl(\rk M-1,M)$ and hence with a $(\rk M-1)$ dimensional cone $\sigma_{\Ff}=\sum_{i=1}^{\rk M-1} \RR_{\geq 0}\overline{e_{F_i}}$ in the fan $\Sigma_n$. There is a well-defined linear function
\[\CH^{\rk M-1}(X_{\perm_n},\QQ)\rightarrow \QQ\]
that to a class represented by a $(\rk M-1)$ dimensional cone $\sigma$ associates one if there exists $\Ff\in \FlFl(\rk M-1,M)$ such that $\sigma=\sigma_{\Ff}$ and zero otherwise. This linear function, via the perfect pairing described in point (4), Proposition \ref{prop:ChowOFperm}, may be identified with a class in $\CH^{n-\rk M+1}(X_{\perm_n},\QQ)$, which we denote by $[X_M]$. 

In other words the class $[X_M]$ pairs to one with $[\sigma_{\Ff}]$ for $\Ff\in \FlFl(\rk M-1,M)$ and to zero with other cones of dimension $\rk M-1$. 
\end{defi}

\begin{rem}
The fact that $[X_M]$ is well-defined is not obvious, as the classes $[\sigma]$ do not form a linear basis of the Chow ring, but only span it. The interested readers are referred to the construction of the Bergman fan \cite{ardila2006bergman, sturmfels2002solving, katz2011realization, June2}, Minkowski weights \cite{fulton1997intersection}, relations to tropical geometry \cite{JuneRev} or a very general construction of equivariant vector bundles \cite{berget2023tautological}. In the last article $[X_M]$ is realized as the top Chern class of a very explicit vector bundle on $X_{\perm_n}$. We briefly sketch this last equivariant construction in Section \ref{subsect:equivariant}.     
\end{rem}

\subsection{Permutohedron and the variety of complete collineations}\label{sec:permCC}

One may identify $\PP^n$ with the projectivisation of the space of $(n+1)\times (n+1)$ diagonal matrices. In this way, one may define $X_{\perm_n}$ as the sequence of blow-ups, where at step $i$ we blow-up the strict transform of rank $i$ matrices. A related construction is that of the variety of complete quadrics $\CQ_n$ or complete collineations $\CCl_n$. Here one starts with the space $\PP^{\binom{n+2}{2}-1}$ of symmetric matrices for the complete quadrics and with $\PP^{(n+1)^2-1}$ of general matrices for the complete collineations\footnote{The construction of the variety of complete collineations is even more general, allowing for rectangular matrices.} and performs blow-ups of strict transforms of rank one, rank two, \dots, rank $n$ matrices. Note that contrary to the case of diagonal matrices, each blown-up locus for $\CQ_n$ and $\CCl_n$ is an irreducible variety. We obtain the diagram:

\begin{center}
\begin{tikzcd}
\prod_{i=1}^{n} \PP^{\binom{n+1}{i}-1}  \arrow[hookrightarrow]{r}           & \prod_{i=1}^{n} \PP^{\binom{n+1}{i}\left(\binom{n+1}{i}+1\right)/2-1} \arrow[hookrightarrow]{r}& \prod_{i=1}^{n} \PP^{\binom{n+1}{i}^2-1}\\
      X_{\perm_n}\arrow[hookrightarrow]{r}\arrow[hookrightarrow]{u} \arrow[twoheadrightarrow]{d}& \CQ_n\arrow[hookrightarrow]{u} \arrow[twoheadrightarrow]{d}\arrow[hookrightarrow]{r}           &\CCl_n\arrow[hookrightarrow]{u}\arrow[twoheadrightarrow]{d}\\
      \PP^n \arrow[hookrightarrow]{r}&        \PP^{\binom{n+2}{2}-1}\arrow[hookrightarrow]{r}     &\PP^{(n+1)^2-1}\\
\end{tikzcd}
\end{center}

Here the upper vertical inclusion maps are defined on the open set of full-rank matrices by taking all $i\times i$ minors of a given matrix. It turns out that taking all $i$ from $1$ to $n$ one obtains a regular embedding\footnote{For the middle inclusion of the variety of complete quadrics it is possible to take projective spaces of smaller dimension, as there are linear relations among minors of a generic symmetric matrix, which correspond to Pl\"ucker relations.}. The surjective arrows are iterative blow-up maps. In this way we obtain exceptional divisors $S_{i,\CQ_n}$ and $S_{i,\CCl_n}$, which in case of $\CQ_n$ and $\CCl_n$ generate the rational Picard group. Via the middle row, the pull-back of $S_{i,\CCl_n}$ is respectively $S_{i,\CQ_n}$ and $S_i$. The upper vertical inclusions, via pull-back of hyperplane class, give us divisors $L_{i,\CQ_n}$ and $L_{i,\CCl_n}$. Again, by considering the middle row, the pull-back of $L_{i,\CCl_n}$ is respectively $L_{i,\CQ_n}$ and $L_i$.

Classically, the number of smooth quadrics in $\PP^n$ passing through $a_1$ general points, tangent to $a_2$ general lines, tangent to $a_3$ general planes, $\dots$, and tangent to $a_n$ general hyperplanes equals $\int_{\CQ_n} \prod_{i=1}^n L_{i,\CQ_n}^{a_i}$. Schubert introduced a technique to compute this number by translating $L_{i,\CQ_n}$'s to linear combinations of $S_{i,\CQ_n}$. This was advantageous, as the divisor $S_i$ has a birational model $\PP(S^2\mathcal{U}^*)\times_{G(i,\CC^{n+1})}\PP(S^2\mathcal{Q})$, where $\mathcal{U}$ is the tautological bundle on the Grassmannian $G(i,\CC^{n+1})$ and $\mathcal{Q}$ is the quotient bundle\footnote{More precisely $S_i$ is the subvariety obtained by replacing each projectivised bundle, by a relative complete quadrics construction, which is possible for vector bundles.}. 
The double fiber bundle construction, restricted to the permutohedral variety gives back the well-known result that faces of the permutohedron are products of smaller permutohedra. We will thus follow the strategy introduced by Schubert, but in the setting of $X_{\perm_n}$. In this setting matroidal mixed Eulerian numbers for representable matroids are exactly the case of characteristic numbers of tensors represented by linear spaces of diagonal matrices, as defined in \cite{conner2021characteristic}. 

The Picard group of $\CQ_n$ may be understood as follows. The intersection of all exceptional divisors $S_{i,\CQ_n}$ is a full flag variety $Y$. The inclusion $Y\subset \CQ_n$ induces an injective morphism of groups $\Pic(\CQ_n)\rightarrow \Pic(Y)$. The right hand side may be identified with the lattice of a type $A$ root system. Under this identification $S_{i,\CQ_n}$ correspond to (twice) the simple positive roots $e_i-e_{i+1}$ and $L_{i,\CQ_n}$ to (twice) the fundamental roots $e_1+\dots+e_i$. This also explains the form of relations \eqref{eqn:lineartransLiSi}. From this perspective it is also natural to consider analogues of $e_i$, which is equivalent to working in the basis $L_{i+1, \CQ_n}-L_{i,\CQ_n}$. Indeed, in the setting of complete quadrics these divisors play an important role \cite[Section 4]{LaksovLascouxThorup}, where they correspond to $\mathcal{L}_i$ (while $L_{i,\CQ_n}$ correspond to $\mathcal{M}_i$ and $S_{i,\CQ_n}$ correspond to $\mathcal{J}_i$).  

For more information about $\CQ_n$ and $\CCl_n$ we refer to \cite{ThaddeusComplete, LaksovLascouxThorup}. For expository articles presenting in more details interactions among the constructions above we refer to \cite{michalek2023enumerative, dinu2021applications}. Finally, for applications of aforementioned techniques outside of combinatorics we refer to \cite{ja1, ja2, MRV, conner2023sullivant}.

 \subsection{Equivariant Chow ring of the permutohedral variety} \label{subsect:equivariant}
 Many of the stated results have a nicer description in the setting of equivariant cohomology, that we briefly describe next. For details on equivariant cohomology we refer to \cite{anderson2012introduction}. There are two natural tori acting on $X_{\Perm_n}$. One is the $(n+1)$-dimensional torus $T:=(\C^*)^{n+1}$ with character lattice $M$ with a basis $t_0,\dots, t_n$ and one-parameter subgroup lattice $N$. The action of $T$ on $X_n$ has a stabilizer given by $\lambda:=\C^*\simeq \{(t,\dots,t)\in T\}$. Thus, the second torus is 
 $\hat T:=T/\lambda \cong (\C^*)^n$ with one-parameter subgroup lattice $\hat N=N/\langle (1,\dots,1)\rangle$ and character lattice $(\hat N)^*=:\hat M\subset M$ defined by the hyperplane with the defining equation summing all of the coordinates. 
 It is the second torus $\hat T$ that may be regarded as a dense subset of the permutohedral variety and makes the variety toric in the sense of classical toric geometry. 

\begin{rem}
On the projective space $\PP^n$ we also have the action of $T$ and $\hat T$. While if we just look at the toric geometry of $\PP^n$ it is more natural to work with $\hat T$ the situation changes when we look at line bundles and their sections. For example, there is a natural linearization of the action of $T$ on $\mathcal{O}(1)$ --- we may simply rescale the variable $x_i$ by $t_i^{-1}$. However, there is no natural linearization for the action of $\hat T$. For this reason there is no distinguished piecewise linear function on the fan of $\PP^n$ representing $\mathcal{O}(1)$ --- we have to make a choice. Still, when working with equivariant cohomology with respect to $T$ there is a natural representative. 
\end{rem}    

The permutohedron $\Perm_n$ is naturally included in $M_\RR$. Each edge $e$ has two vertices $v_1,v_2\in M$. By choosing (any) orientation of $e$ we associate to it a character $w(e):=v_1-v_2\in \hat M\subseteq M$. We regard $w(e)$ as a formal linear combination of $t_i$'s, i.e.~as a homogeneous linear function in $t_i$'s. 
The next theorem is a consequence of the Chang-Skjelbredt Lemma and describes the equivariant Chow ring $\CH_T(X_{\Perm_n},\QQ)$.
\begin{thm}
Let $\mathfrak{S}_{[n]}$ be the set of permutations of $[n]$ that we identify with the vertices of $\Perm_n\subset M_{\RR}$.
The equivariant Chow ring $\CH_T(X_{\Perm_n},\QQ)$ is a subring of $(\mathfrak{S}_{[n]})^{\QQ[t_0,\dots,t_n]}$, i.e.~associations of polynomials $f_v\in \QQ[t_0,\dots,t_n]$ to vertices $v$ of $\Perm_n$. Such an association $(f_v)$ belongs to $\CH_T(X_{\Perm_n},\QQ)$ if and only if for every edge $e=(v_1,v_2)$ the linear function $w(e)$ divides $f_{v_1}-f_{v_2}$.
\end{thm}
\begin{rem}
Given an element $(f_v)\in \CH_T(X_{\Perm_n},\QQ)$ even if all $f_v$ are homogeneous of degree one it may not be possible to identify it with a piecewise linear function on the normal fan of $\Perm_n$. The reason is that such $f_v$ belongs to $M$, but not to $\hat M$. The compatibility of $f_v$'s along edges forces that the sum of coordinates of each $f_v$ does not depend on $v$. Thus, one could shift all $f_v$ by a multiple of e.g.~$t_n$ to obtain such a piecewise linear function on the fan. However, different shifts lead to different piecewise linear functions. Each two of them differ by a globally linear function. 
\end{rem}
\begin{thm}
Let $I$ be the ideal of  $\CH_T(X_{\Perm_n},\QQ)$  generated by all elements $(f_v)$, where $f_v=f$ does not depend on $v$. The quotient $\CH_T(X_{\Perm_n},\QQ)/I$ is isomorphic to the classical Chow ring $\CH(X_{\Perm_n},\QQ)$ and we have a natural map $\pi_T: \CH_T(X_{\Perm_n},\QQ)\rightarrow \CH(X_{\Perm_n},\QQ)$.
\end{thm}
Following \cite{berget2023tautological}, the setting above allows us to define equivariant classes $[X_M]_T\in  \CH_T(X_{\Perm_n},\QQ)$ that descend to $\pi_T([X_M]_T)=[X_M]$.

\begin{defi}
    Let $M$ be a matroid on $[n]$. For $1\leq i\leq \corank(M)$ we define the class $[\M_i]_T\in \CH_T(X_{\Perm_n},\QQ)$ of an equivariant line bundle as follows:
    \begin{itemize}
        \item For each vertex $v$ of $\Perm_n$ corresponding to a permutation $\pi$ of $[n]$ let $a_v=(a_{v,1},\dots,a_{v,\corank(M)})$ be the complement of the lex-first basis of $M$ in the order of $[n]$ induced by the permutation $\pi$.
        \item Define $[\M_i]_T=(f_v)$, where $f_v=-t_{a_{v,i}}$ corresponds to the $i$-th element of a.
    \end{itemize}
    The class $[X_M]_T$ is the top Chern class of the vector bundle that is a direct sum of $T$-equivariant line bundles corresponding to all $[\M_i]_T$. Explicitly, using the notation above, we have $[X_M]_T=(\tilde{f}_v)$ where $\tilde{f}_v=(-1)^{\corank(M)}\prod_{i=1}^{\corank(M)} t_{a_{v,i}}$.
\end{defi}
The representation of $[X_M]$ as a top Chern class of a vector bundle that is a sum of line bundles, provides a representation of $[X_M]$ as a product of divisors. We obtain:
\begin{align}
\label{factorization}
    [X_M]= \prod_{i=1}^{\corank(M)}[\M_i] \in \CH_T^{\corank(M)}(X_{\Pi_n},\Q).
\end{align}  
We recall that over a vertex $v$ of $\Perm_n$ the line bundle $[\M_i]_T$ is represented by a character $-t_{a_i}$ of the lattice $M$. Thus to obtain a representation of the divisor $[\M_i]$ we need to choose a shift of all such representatives by some fixed character to obtain a character of $\hat M$. There are many different ways to do so and all of them will give the same result after we divide by the ideal that is the kernel of $\pi_T$. For now let us add the global linear function $t_{0}$. The resulting class has the function $t_{0}-t_{a_i}$ over vertex $v$ of $\Perm_n$, which we identify with a maximal cone of the normal fan.\par
 To map to $\CH(X_{\Pi_n},\QQ)$ we simply need to evaluate these functions on the rays of the permutohedral fan and use the results as coefficients of a class in the Chow ring. Recall that rays of the normal fan are in natural bijection with facets of $\Perm_n$ and hence with nonempty proper subsets $S\subset [n]$.

We obtain $[\M_i]=\sum_{\emptyset \subsetneq S\subsetneq [n]}c_{S,i}x_S\in \CH^1(X_{\Pi_n},\QQ)$ for the coefficients \begin{align}
   \label{coefficients_Mi}
       c_{S,i}= \begin{cases}
           -1 &\text{, if } e_0\notin S, \nullity_M(S)\geq i \\
           0 &\text{, if } e_0\notin S, \nullity_M(S)< i \\
           0 &\text{, if } e_0\in S, \nullity_M(S)\geq i \\
           1 &\text{, if } e_0\in S, \nullity_M(S)< i \\
       \end{cases}.
   \end{align}

   If needed, to emphasize the dependence on the matroid $M$, we write $[\M_{i,M}]$ for $[\M_i]$.

    For more information about equivariant classes of matroids we refer to \cite{berget2023log, berget2023tautological, bauer2023equivariant, galgano2024equivariant}.
    \subsection{Computing intersection numbers}
\label{computing}
In this section we will develop the main technique that we will use to compute matroidal mixed Eulerian numbers. This trick was already used by Schubert to compute intersection numbers on the variety of complete quadrics and Grassmanians even before a formal theory of Chow rings was developed. Say we want to compute an intersection product \begin{align*}
    A_M(c_1,\ldots,c_n)= \int_{X_{\Perm_n}}\ctop \cdot L_1^{c_1}\cdots L_n^{c_n} 
\end{align*}
where $M$ is a matroid on $E:=[n]$ and $c_1,\ldots,c_n \in \N_0$ satsify $\sum_{i=1}^nc_i=\rank(M)-1$. The technique consists of two steps. First we find one divisor $L_i$ that appears with nonzero exponent in the product, say for simplicity $c_1>0$. Via (\ref{LfromS}) we can write $L_1$ as a $\QQ$-linear combination of the exceptional divisors $S_1,\ldots,S_n$. Since the degree map is $\Q$-linear we now only need to compute the numbers \begin{align*}
    \int_{X_{\Pi_n}}\ctop \cdot S_i \cdot L_1^{c_1-1}L_2^{c_2}\cdots L_n^{c_n}
\end{align*}
for every $i=1,\ldots,n$.
The second step of the technique is to realize intersection with $S_i$ as passing to a product of two smaller permutohedra. Therefore we can recur to solving two intersection problems but both of them on strictly smaller permutohedra. This allows us to compute matroidal mixed Eulerian numbers inductively. We will for now further split up the divisors $S_i=\sum_{\substack{F\subsetneq [n]\\ |F|=n+1-i}}x_F$ and restrict to the divisors corresponding to each $x_F$ instead. However, all such restrictions will give isomorphic varieties, so one can still group them according to the divisors $S_1,\ldots,S_n$. \par
Since in the recursive process, the ground set $E=[n]$ will change, we shall from now on keep track more closely of the ground set we are currently working with. We will therefore write $\Pi_E$ and $X_{\Pi_E}$ instead of $\Pi_n$ and $X_{\Pi_n}$. Similarly for any subset $F\subseteq E$ we will write $\Pi_F$ and $X_{\Pi_F}$ for $\Pi_{|F|-1}$ and $X_{\Pi_{|F|-1}}$ to indicate that the vertices of the permutohedron $\Pi_F$ are naturally labelled by permutations of the set $F$ rather than permutations of $[|F|-1]$. This is purely a notational change, $\Pi_F$ is still formally the convex hull of the points obtained by acting on $(0,\ldots,|F|-1)$ with permutations on $|F|$ elements. \par

Let us start by understanding how to intersect with a divisor corresponding to a single ray $x_F, \emptyset \subsetneq F \subsetneq E$ on the permutohedral variety. The facet of the permutohedron that is normal to $\overline{e_F}$ is a product of two smaller permutohedra of dimensions $|F|-1$ and $n-|F|$. Indeed we know that the facet normal to $\overline{e_F}$ is the convex hull of all points $p_\sigma:=(\sigma^{-1}(0),\ldots,\sigma^{-1}(n))-\frac{n(n+1)}{2}t_n$ where $\sigma(0,\ldots,|F|-1)=F$. Clearly \begin{align*}
    \conv(p_\sigma| \sigma(0,\ldots,|F|-1)=F) = \conv( p_\tau \mid \tau\in \mathfrak{S}_{|F|})\times \conv(p_\mu \mid \mu\in \mathfrak{S}_{n+1-|F|})
\end{align*} 
where we separate coordinates into the two sets $F$ and $E\setminus F$.
We can therefore interpret intersection with $x_F$ as restriction to a product of two permutohedral varieties of smaller dimension, $X_{\Pi_F}\times X_{\Pi_{E\setminus F}}$. We can understand the Chow ring of such a product via the Künneth formula for cohomology \cite[Theorem 5.5.11]{spanier} which implies that
\begin{align*}
    H^k(X_{\Pi_{F}}\times X_{\Pi_{E\setminus F}}, \QQ) \cong \bigoplus_{i+j=k}H^i(X_{\Pi_F},\Q)\otimes_\Q H^j(X_{\Pi_{E\setminus F}},\Q)
\end{align*}
for all $k\in \N_0$. Hence, if we grade the tensor product multiplicatively we obtain\begin{align*}
    &\CH^\bullet(X_{\Perm_F}\times X_{\Perm_{E\setminus F}},\QQ)\cong H^\bullet(X_{\Perm_{F}}\times X_{\Perm_{E\setminus F}}, \QQ) \\&\cong H^\bullet(X_{\Perm_F},\QQ)\otimes_\QQ H^\bullet(X_{\Perm_{E\setminus F}},\QQ) \cong \CH^\bullet(X_{\Perm_F},\QQ)\otimes_\QQ \CH^\bullet(X_{\Perm_{E\setminus F}},\QQ).   
\end{align*}
To keep track of the grading we will write \begin{align*}
    \CH^{(i,j)}(X_{\Pi_F}\times X_{\Pi_{E\setminus F}},\Q):= \CH^i(X_{\Perm_F},\QQ)\otimes_\QQ \CH^j(X_{\Perm_{E\setminus F}},\QQ)\subseteq \CH^{i+j}(X_{\Perm_F}\times X_{\Perm_{E\setminus F}},\QQ).
\end{align*}
Next, we explain the restriction process in practice.  Since the Chow ring of $X_{\Pi_E}$ is generated in degree one, it is enough to describe restrictions for divisors. We describe how to restrict a general divisor $D$ to the two factors of the divisor corresponding to $x_F$. If $D$ is a sum of rays $D= \sum_{\emptyset\subsetneq S\subsetneq E}c_Sx_S$ such that $c_F=0$, then the restriction to $x_F$ is obtained by deleting all rays that do not form a two-dimensional cone with $x_F$ from this sum. Since any divisor is linearly equivalent to one of the upper form, this completely determines how to restrict any divisor. More precisely we make the following definition. 
\begin{defi}
    Let $D=\sum_{\emptyset \subsetneq S\subsetneq E}c_Sx_S\in \CH^1(X_{\Perm_E},\Q)$ be a divisor on the permutohedral variety and fix $\emptyset\subsetneq F\subsetneq E$. Pick $a\in F$ and $b\in E\setminus F$. Writing \begin{align*}
        \Tilde{c}_S=\begin{cases}
            c_S & \text{, if } a,b \notin S \text{ or } a,b\in S \\
            c_S-c_F & \text{, if } a \in S, b\notin S \\
            c_S+c_F & \text{, if } a \notin S, b \in S 
        \end{cases}
            \end{align*}
        we have $D= \sum_{\emptyset \subsetneq S\subsetneq E}\Tilde{c}_Sx_S\in \CH^1(X_{\Perm_E},\QQ)$ and $\Tilde{c}_F=0$.
    We define \begin{align*}
        D|_{\subsetneq F}&:= \sum_{\emptyset \subsetneq S \subsetneq F}\Tilde{c}_Sx_S \in \CH^1(X_{\Perm_F},\QQ)\cong \CH^{(1,0)}(X_{\Perm_F}\times X_{\Perm_{E\setminus F}},\QQ)\subseteq \CH^{1}(X_{\Perm_F}\times X_{\Perm_{E\setminus F}},\QQ) \\
        D|_{F\subsetneq}&:=\sum_{\emptyset \subsetneq S\subsetneq E\setminus F}\Tilde{c}_{S\cup F}x_S \in \CH^1(X_{\Perm_{E\setminus F}},\QQ)\cong \CH^{(0,1)}(X_{\Perm_F}\times X_{\Perm_{E\setminus F}},\QQ)\subseteq \CH^{1}(X_{\Perm_F}\times X_{\Perm_{E\setminus F}},\QQ) 
    \end{align*}
    and finally the \emph{restriction of $D$ to $F$} by \begin{align*}
        D|_F:=D|_{\subsetneq F}\otimes 1+1\otimes D|_{F\subsetneq}\in \CH^1(X_{\Pi_F}\times X_{\Pi_{E\setminus F}},\QQ).
    \end{align*}
    \end{defi}
    All three divisors $D|_F,D|_{\subsetneq F},D|_{F\subsetneq}$ are independent of the choice of $a$ and $b$ and only depend on $D$ up to linear equivalence. Indeed, any linear relation among divisors on $X_{\Pi_E}$ that does not involve $x_F$ restricts to a linear relation among divisors on $X_{\Pi_F}$ and on $X_{\Pi_{E\setminus F}}$ respectively. 
    The following lemma explains why restrictions are useful for computing intersection numbers since it allows us to move our computations to smaller permutohedral varieties. \\
    \begin{lemma}
    \label{pushforward-formulas}
        Let $D_1,\ldots,D_{n-1}\in \CH^1(X_{\Pi_E},\QQ)$ be divisors and fix $\emptyset \subsetneq F \subsetneq E$, then we have the following equalities of integers
        \begin{align*}
            \int_{\Pi_E}x_F\prod_{i=1}^{n-1}D_i = \int_{\Pi_F\times \Pi_{E\setminus F}}\prod_{i=1}^{n-1}D_i|_F= \sum_{\substack{I\subseteq \{1,\ldots,n-1\}\\ |I|=|F|-1}}\int_{\Pi_F}\prod_{i\in I}D_i|_{\subsetneq F}\cdot \int_{\Pi_{E\setminus F}}\prod_{i\in \{1,\ldots,n-1\}\setminus I}D_i|_{F \subsetneq}.
        \end{align*}
    \end{lemma}
\begin{proof}
    By Proposition \ref{prop:ChowOFperm} we can change the divisors $D_i$ up to linear equivalence such that the product on the left turns into a sum over square-free monomials in the generators $x_S, \emptyset \subsetneq S \subsetneq [n]$. Since all three terms above are linear in every $D_i$ we reduce to the case $D_i=x_{F_i}$ for suitable sets $F_i$. If two such sets $F_i,F_j$ are incomparable then the left-hand side is clearly zero. It also follows that at least one of $F_i,F_j$ must be incomparable with $F$ and hence has zero restriction to $F, \subsetneq F, F\subsetneq$. Therefore all three terms in the claimed equality are zero in this case and we can assume that $F_i$ fit into a full flag
    for a full flag \begin{align*}
       \mathcal{F}:=\{\emptyset \subsetneq F_1\subsetneq \cdots \subsetneq F_{|F|-1}\subsetneq F \subsetneq F_{|F|}\subsetneq \cdots \subsetneq F_{n-1}\subsetneq E\}
    \end{align*}
    which passes through $F$.
    In this case by definition the expression on the left is 1 since $\mathcal{F}$ corresponds to a maximal cone in the normal fan of $\Pi_E$. \\
    For the other two expressions note that by definition \begin{align*}
        D_i|_{\subsetneq F}&= \begin{cases}
            x_{F_i} &\text{, if } i <|F| \\
            0 & \text{, else}
        \end{cases},\\
          D_i|_{F \subsetneq }&= \begin{cases}
            x_{F_i\setminus F} &\text{, if } i \geq |F| \\
            0 & \text{, else}
        \end{cases}.
    \end{align*}
    The third term in the statement of the Lemma hence simplifies to a product of two terms, namely \begin{align*}
        \int_{\Pi_F}\prod_{i=1}^{|F|-1}x_{F_i}\int_{\Pi_{E\setminus F}}\prod_{i=|F|}^{n-1}x_{F_i\setminus F}.
    \end{align*}
   Both factors are equal to one, as they correspond to maximal cones in the respective fans.
    
    Finally, the term in the middle also is equal to one since the fan for $X_{\Pi_F}\times X_{\Pi_{E\setminus F}}$ is the product of the two permutohedral fans. More precisely, a cone in the fan of  $X_{\Pi_F}\times X_{\Pi_{E\setminus F}}$ is the cartesian product of a cone in the normal fan of $\Pi_F$ with a cone in the normal fan of $\Pi_{E\setminus F}$. In particular the maximal cones are given by the product of a maximal cone in the normal fan of $\Pi_{F}$ times a maximal cone in the normal fan of $\Pi_{E\setminus F}$, so the product $x_{F_1}\cdots x_{F_{|F|-1}}x_{F_{|F|}\setminus F}\cdots x_{F_{n-1}\setminus F}$ precisely encodes a maximal cone in this fan as well.
\end{proof}
To apply this technique in our setting we need to understand how the divisors $L_1,\ldots,L_n$ and the class $\ctop$ restrict to a ray  $x_F$. For the matroid class $\ctop$ one may use the factorization (\ref{factorization}) and then compute how the factors $[\M_i]$ restrict to $x_F$, see Remark \ref{rem:restrictingM}. One can also derive an explicit formula from \cite[Lemma 2.6]{katz2023matroidal}. Restricting any linear combination of rays that does not contain $x_F$ to $x_F$ is easy, hence it suffices to understand the restriction $x_F|_F$. We carry out this computation in the following lemma. Since we will deal with different permutohedral varieties at the same time, we will for now write $L_{1,E},\ldots,L_{|E|-1,E}$ for the divisors $L_1,\ldots,L_n$ on $X_{\Pi_E}$ to clarify which variety they come from.
\begin{lemma}
\label{restricting x_F}
    For $\emptyset \subsetneq F \subsetneq E$, the divisor $x_F$ restricts to $X_{\Pi_F}\times X_{\Pi_{E\setminus F}}$ as the sum of the following two restrictions. \begin{align*}
        x_F|_{\subsetneq F}&= - L_{1,F} \in \CH^1(X_{\Pi_F},\QQ) \\
        x_F|_{F\subsetneq} &= - L_{n-|F|,E\setminus F} \in \CH^1(X_{\Pi_{E\setminus F}}, \QQ)
    \end{align*}
    In particular for any divisor $D= \sum_{\emptyset\subsetneq S \subsetneq E}c_Sx_S\in \CH^1(X_{\Pi_E},\QQ)$ we have \begin{align*}
        D|_{\subsetneq F}&= \sum_{\emptyset \subsetneq S \subsetneq F}c_Sx_S-c_FL_{1,F} \in \CH^1(X_{\Pi_F},\QQ) \\
                D|_{F\subsetneq }&= \sum_{F \subsetneq S \subsetneq E}c_Sx_{S\setminus F}-c_FL_{n-|F|,E\setminus F} \in \CH^1(X_{\Pi_{E\setminus F}},\QQ) \\
    \end{align*}
\end{lemma}
\begin{proof}
    Pick two elements $a\in F, b \in E\setminus F$. We have 
    \[\sum_{S\ni a}x_S=\sum_{S\ni b}x_S\in \CH^1(X_{\Pi_E},\Q)\] and so \begin{align*}
        x_F|_{\subsetneq F}&= \left(x_F-\sum_{S\ni a}x_S+\sum_{S\ni b}x_S\middle)\right|_{\subsetneq F} = - \sum_{a\in S \subsetneq F}x_S = -L_{1,F} \\
        x_F|_{F\subsetneq }&= \left(x_F-\sum_{S\ni a}x_S+\sum_{S\ni b}x_S\middle)\right|_{\subsetneq F} = - \sum_{ S \supsetneq F}x_{S\setminus F} + \sum_{b \in S \supsetneq F}x_{S\setminus F}=-\sum_{b \notin S \supsetneq F}x_{S\setminus F}=\\&=-L_{n-|F|,E \setminus F}.
    \end{align*}
    Here the second equalities in both rows follow since $a \in F, b \notin F$, thus being contained in $F$ implies not containing $b$ (first row) and containing $F$ implies containing $a$ (second row). The last equalities in both lines are due to Example \ref{exampleL1Ln}.
\end{proof}

The following lemma is a direct consequence of \cite[Lemma 2.6]{katz2023matroidal} where the classes $\gamma$ correspond to our classes $L$.
\begin{lemma}
\label{restrictions of Li}
    For $\emptyset \subsetneq F \subsetneq E$ and $i\in \{1,\ldots,n\}$ we have \begin{align*}
        L_{i,E}|_{\subsetneq F}&= \begin{cases}
            L_{i-(n+1-|F|),F} & \text{ if } i>n+1-|F| \\
            0 & \text{ else}
        \end{cases}, \\
        L_{i,E}|_{F\subsetneq} &= \begin{cases}
            L_{i,E\setminus F} &\text{ if } i < n+1-|F| \\
            0 &\text{ else}
            \end{cases}.
    \end{align*}
\end{lemma}
Finally, we recall what the restriction of the class $\ctop$ is. The next lemma is a direct consequence of \cite[Proposition 5.3]{berget2023tautological}.
\begin{lemma}
\label{restrictions of ctop}
Let $M$ be a matroid on $E$ and let $F$ be a non-empty proper subset of $E$. We have:
\begin{align*}
        \ctop|_{F}= (\ctop|_{M|F}])\otimes (\ctop|_{M/F})\in \CH(X_{\Pi_F},\QQ)\otimes_\QQ \CH(X_{\Pi_{E\setminus F}},\QQ).
    \end{align*}
\end{lemma}

\begin{rem}\label{rem:restrictingM}
Lemma \ref{restrictions of ctop} may also be proved by noting that
    for any $i\in \{1,\ldots,\rank(M)\}$ we have \begin{align*}
        [\M_{i,M}]|_{\subsetneq F} &= \begin{cases}
            [\M_{i,M|F}] &\text{, if } i\leq \nullity_M(F) \\
            0 &\text{, else}
        \end{cases},\\
          [\M_{i,M}]|_{F \subsetneq} &= \begin{cases}
            [\M_{i,M/F}] &\text{, if } i> \nullity_M(F) \\
            0 &\text{, else}
        \end{cases}
    \end{align*}
    and applying \eqref{factorization}.
\end{rem}

\section{Main Theorems}
\subsection{The span of the hypersimplex classes $L_i$}
We start by recalling an explicit description of the $\QQ$-subalgebra of $\CH(X_{\Perm_n},\QQ)$ generated by $L_1,\dots,L_n$ or equivalently $S_1,\dots, S_n$. This is the $\SS_{n+1}$-invariant subring that is also isomorphic to the cohomology ring of the Peterson variety and known as Klyachko algebra \cite{abe2022integral, fukukawa2015equivariant, klyachko1985orbits}.
The lemma implies that knowing all intersection numbers of $\ctop$ with the $L_i$'s is equivalent to knowing the symmetrized class $\sum_{\sigma\in \SS_{n+1}} \sigma\cdot \ctop$. 
\begin{lemma}
\label{invariant part}
The $\QQ$-subalgebra of $\CH(X_{\Perm_n},\QQ)$ generated by the divisors $S_1,\ldots,S_n$ equals $\CH(X_{\Perm_n},\Q)^{\SS_{n+1}}$. In particular, the linear span of $S_1,\dots,S_n$ is 
     $\left(\CH^1(X_{\Perm_n},\QQ)\right)^{\SS_{n+1}}$.
\end{lemma}
\begin{proof}
It is clear that $S_1,\ldots,S_n$ are $\SS_{n+1}$-invariant and therefore they generate a $\QQ$-subalgebra of  $\CH(X_{\Perm_n},\QQ)^{\SS_{n+1}}$. On the other hand take any $D\in \CH(X_{\Pi_n},\QQ)^{\SS_{n+1}}$ and write it as a $\QQ$-linear combination of square-free monomials as in Proposition \ref{prop:ChowOFperm}, say \begin{align*}
        D= \sum_{i=1}^k q_{\mathcal{F}_i}x_{\mathcal{F}_i}.
    \end{align*}
    Here $\mathcal{F}_i$ for $i=1,\ldots,k$ is a flag of proper non-empty subsets of $[n]$ and for $\mathcal{F}_i: \emptyset \subsetneq F_{i,1}\subsetneq \cdots \subsetneq F_{i,l_i}\subsetneq [n]$ we have $x_{\mathcal{F}_i}:=x_{F_{i,1}}\cdots x_{F_{i,l_i}}$ and $q_{\mathcal{F}_i}\in \Q$. Since $D$ is $\SS_{n+1}$-invariant we have \begin{align*}
        D= \frac{1}{(n+1)!}\sum_{\sigma \in \mathfrak{S}_{n+1}}\sigma\cdot D
        =\frac{1}{(n+1)!}\sum_{i=1}^k q_{\mathcal{F}_i}\sum_{\sigma \in \mathfrak{S}_{n+1}}\sigma\cdot x_{\mathcal{F}_i}
    \end{align*}
    It is therefore enough to show that for any flag $\mathcal{F}:\emptyset\subsetneq F_1\subsetneq\cdots \subsetneq F_l \subsetneq [n]$ the Chow class \begin{align*}
        m:=\sum_{\sigma \in \mathfrak{S}_{n+1}}\sigma\cdot x_{\mathcal{F}}
    \end{align*}
    is in the $\QQ$-subalgebra generated by $S_1,\ldots,S_n$. For $i=1,\ldots,l$ we set $a_i:=n+1-|F_i|$ and claim that $m=q S_{a_1}\cdots S_{a_l}$ for
    \[
    q:=(n+1-a_1)!(a_1-a_2)!(a_2-a_3)!\cdots (a_{l-1}-a_l)!.
    \]
    
    
    An arbitrary monomial appearing in $m$ has the form
        $x_{T_1}\cdots x_{T_l}$ where $|T_i| = |F_i|$ and $\emptyset \subsetneq T_1 \subsetneq \dots \subsetneq T_l \subsetneq [n]$. Every such monomial appears in $m$ with the same coefficient $q$, as this is the number of permutations fixing $(T_1,\dots,T_l)$. On the other hand, by unwrapping the product \begin{align*}
            S_{a_1}\cdots S_{a_l} = \left(\sum_{\substack{\emptyset \subsetneq S\subsetneq [n] \\ |S|=n+1-a_1}}x_S\right)\cdots \left(\sum_{\substack{\emptyset \subsetneq S\subsetneq [n] \\ |S|=n+1-a_l}}x_S\right)
        \end{align*}
        and by using that the product of two variables labeled by incomparable sets is zero in $\CH(X_{\Perm_n},\QQ)$,
        we see that the same monomials appear as in $m$ but their coefficient is one. Hence $m=qS_{a_1}\cdots S_{a_l}$ as claimed.
\end{proof}

    Using the results of \cite{abe2022integral} the previous lemma may be easily lifted to integral cohomology.
    \begin{lemma}
    We have an isomorphism of three rings:     
    \begin{enumerate}
        \item the subring of $\CH(X_{\Pi_n})$ generated by the classes $L_i$,
        \item the integral cohomology ring of the Peterson variety and 
        \item the $\SS_{n+1}$ invariant part of the integral cohomoloy ring $\CH(X_{\Pi_n})$.
    \end{enumerate}
    \end{lemma}
\begin{proof}
    By \cite[Theorem 1.1]{abe2022integral} we know that the last two rings are isomorphic. Further, by \cite[Theorem 1.2]{abe2022integral} these two rings are generated in degree one (or two---depending on the grading convention). As the $L_i$ are $\SS_{n+1}$ invariant, the first subring is contained in the third one. Thus to finish the proof, it is enough to note that the $L_i$ integrally span the degree one  part of the Peterson variety via isomorphism \cite[Theorem 1.1]{abe2022integral}. Indeed, the generators provided in \cite[Theorem 1.2]{abe2022integral} are simply differences of two consecutive $L_i$'s. 
\end{proof}
\begin{rem}
    We note that the exceptional divisors $S_i$ span the same subring over rational rings, by relations \eqref{eqn:lineartransLiSi}. However, the previous lemma fails if we replace the nef divisors $L_i$ by $S_i$ as over integers the divisors $S_i$ do not generate the $L_i$.
\end{rem}

In the next lemma we show that the divisors $L_i$ not only span the $\SS_{n+1}$ invariant part of the Picard group, but are the ray generators of the intersection of the nef cone with the space of $\SS_{n+1}$ invariants. For details on nef cones of toric varieties see e.g. \cite[Chapter 6]{CoxLittleSchenck}.

\begin{lemma}
\label{lem:LDivisorsspanNefCone}
     The divisors $L_i$ minimally generate the cone that is the intersection of the nef cone with the invariant Chow ring $\CH(X_{\Perm_n},\QQ)^{\SS_{n+1}}$.
\end{lemma}
\begin{proof}
As each $L_i$ is a pull-back of an ample divisor, it is nef. It is also $\SS_{n+1}$-invariant. Thus it is enough to prove that every $\SS_{n+1}$ invariant nef divisor is a nonnegative $\Q$-linear combination of the $L_i$.
We present two proofs of this statement.
\begin{proof}[First proof]
The first proof relies on known properties of permutohedra.

 Consider any rational, $\SS_{n+1}$ invariant, nef divisor $D$. In toric geometry \cite{CoxLittleSchenck} it is represented by a piecewise linear convex function $f_D$ on the normal fan $\Sigma_{\Perm_n}$ of the permutohedron and to a polyope $P_D$. Since $D$ is nef the vertices of $P_D$ are precisely the linear functions that constitute $f_D$ as elements of the vector space $M_\RR$. As $\SS_{n+1}$ acts transitively on the cones of $\Sigma_{\Perm_n}$ and as $D$ is $\SS_{n+1}$ invariant it also acts transitively on the vertices of $P_D$. Thus $P_D$ is the convex hull of an $\SS_{n+1}$ orbit of one point. By \cite[Section 16]{postnikov2009permutohedra} we can present $P_D$ as a Minkowski sum of nonnegative scalings of hypersimplices. This exactly means that $D$ is a nonnegative combination of of the divisors $L_i$.
\end{proof}

\begin{proof}[Second proof]
In the second proof for any divisor $D$ outside of the cone generated by $L_i$'s we explicitly constructs a curve that intersects $D$ negatively. 

    By Lemma \ref{invariant part} the $L_i$ form a basis of the vector space $\CH^1(X_{\Perm_n},\Q)^{\SS_{n+1}}$. Thus, it is sufficient to prove that $D:=\sum_{i=1}^{n} a_i L_i$ is not nef when some $a_{i_0}<0$. The intersection of $n-1$ divisors corresponding to: \[x_{[0]},x_{[1]},\dots, x_{[i_0-1]}, x_{[i_0+1]}, \dots, x_{[n-1]}\]
is isomorphic to a one dimensional permutohedral variety, i.e.~$\PP^1\subset \Perm_n$. Using \ref{restrictions of Li} and \ref{pushforward-formulas} we see that the degree of the intersection of the divisor $D$ with this curve equals $a_{i_0}<0$, which indeed proves that $D$ is not nef. This finishes the proof of the lemma.
\end{proof}
\end{proof}
\subsection{Explicit and recursive formulas for matroidal mixed Eulerian numbers}


Next we present the proof of the formula for the matroidal mixed Eulerian numbers from Theorem \ref{explicit matroid formula}.
\begin{thm}\label{thm:matroidalmixed}
    Let $M$ be a rank $r+1$ matroid on $E=[n]$ and let $\bfa = (a_1,\dots,a_n) \in \N_0^n$ satisfy $\sum_{i=1}^n a_i = r$. Then the associated matroidal mixed Eulerian number $A_M(a_1,\ldots,a_n)$ equals
	\[
       \smashoperator[r]{\sum_{S_\bfb^\bfc \in \Mon^{r}(S_1,\dots,S_n)}} C_{\bfb}^{\bfc}(\bfa) (-1)^{r-k} \smashoperator[l]{\sum_{\substack{\Ff\in \FlFl(k,M) \\ \forall i=1,\ldots,k : \\ |F_i|=n+1-b_{k+1-i} 
       }}} \prod_{i=1}^{k} \gamma_{(M | F_{i+1})/F_{i}}\left( \rk_M(F_{i+1}) - 1 - \sum_{l=1}^i c_{k+1-l} \right).  
    \]
where $k = \leng(\bfb)$ and for any matroid $N$, $\gamma_N(l) = 0$ whenever $l < 0$ or $l > \rk(N)-1$.
\end{thm}
\begin{proof}
    By the definition of the coefficients $C_{\bfb}^{\bfc}(\bfa)$ from equation \eqref{LfromS} it suffices to show that for any $0=b_0<b_1<\ldots< b_k<b_{k+1}=n+1$ and $c_1,\ldots,c_k>0$ with $\sum_{l=1}^k c_l= r$ we have
    \[
    \int_{X_{\Pi_E}} [X_M]S_{b_1}^{c_1}\cdots S_{b_k}^{c_k} = (-1)^{r-k} \smashoperator[l]{\sum_{\substack{\Ff\in \FlFl(k,M) \\ \forall i=1,\ldots,k : \\ |F_i|=n+1-b_{k+1-i} 
    }}} \prod_{i=1}^{k} \gamma_{(M | F_{i+1})/F_{i}}\left( \rk_M(F_{i+1}) - 1 - \sum_{l=1}^i c_{k+1-l} \right).  
    \]
    To prove this we will proceed by induction on $k$ to show the following stronger statement. For all $k,b_1,\ldots,b_k$ as above and $c_1,\ldots,c_{k}>0$ with $\sum_{l=1}^{k}c_l \le r$ we have \begin{multline}
    \label{intersectingS}
        \int_{X_{\Pi_E}}[X_M]S_{b_1}^{c_1}\cdots S_{b_k}^{c_k}L_1^{r-\sum_{l=1}^k c_l} \\= (-1)^{\sum_{l=1}^kc_l-k}\smashoperator[l]{\sum_{\substack{\Ff\in \FlFl(k,M) \\ \forall i=1,\ldots,k : \\ |F_i|=n+1-b_{k+1-i} 
        }}} \prod_{i=1}^{k} \gamma_{(M | F_{i+1})/F_{i}}\left( \rk_M(F_{i+1}) - 1 - \sum_{l=1}^i c_{k+1-l} \right).
    \end{multline}

    \emph{Base case $k=0$:} If $M$ has a loop, then the sum on the right is empty by our definition of $\FlFl(k,M)$. Also in this case $[X_M]=0\in \CH(X_{\Perm_E})$, so equality holds. Otherwise, since $k=0$ the sum on the right goes over precisely one flag of flats, namely $\emptyset \subseteq E$, and the product has no factors so it equals 1. The left hand side of the equation is $\int_{X_{\Pi_E}}\ctop L_1^{\rank(M)-1} = \gamma_M(\rank(M)-1) = 1$, so the equation is true in this case.
    
    \emph{Inductive step $k-1 \to k$:} We will write out one copy of $S_{b_1}$ as a sum of rays and restrict all appearing classes in the product on the left hand side of \eqref{intersectingS} to the appearing rays. We will use the notation $S_{k,E}$ for the divisor $S_k$ on the permutohedral variety $X_{\Pi_E}$ to keep track of the different varieties. The restrictions of $\ctop$ and $L_n$ to any ray $x_F$ were already computed in Lemma \ref{restrictions of ctop} and Lemma \ref{restrictions of Li}. We now need to understand how $S_{b_i,E}$ restricts to $x_F$ for all $i \leq k$ and $F\subseteq E$ with $|F|=n+1-b_1$. If $i>1$, then $n+1-b_i < n+1-b_1$, so 
    \begin{align*}
        S_{b_i,E}|_{\subsetneq F}&=S_{b_i-b_1,F} \\
        S_{b_i,E}|_{F\subsetneq }&=0.
    \end{align*}
    Furthermore, sets of the same cardinality are equal if and only if they are comparable. Thus, by applying Lemma \ref{restricting x_F}, we get that $S_{b_1,E}$ restricted to $x_F$ is \begin{align*}
        S_{b_1,E}|_{\subsetneq F}&=x_F|_{\subsetneq F}= -L_{1,F} \\
        S_{b_1,E}|_{F\subsetneq }&=x_F|_{F\subsetneq }= -L_{b_1-1,E\setminus F}. \\
    \end{align*}
    Therefore applying \ref{pushforward-formulas} we get \begin{align}
    \label{inductionstep}
      \nonumber  &\int_{X_{\Pi_E}}[X_M]S_{b_1}^{c_1}\cdots S_{b_k}^{c_k}L_1^{r-\sum_{l=1}^k c_l}= \int_{X_{\Pi_E}} \sum_{F:|F|=n+1-b_1} x_F[X_M]S_{b_1}^{c_1-1} S_{b_2}^{c_2}\cdots S_{b_k}^{c_k}L_1^{r-\sum_{l=1}^k c_l}
      \\ \nonumber &= 
        \begin{multlined}[t] \sum_{F:|F|=n+1-b_1} \int_{X_{\Pi_F}\times X_{\Pi_{E\setminus F}}}[X_{M|F}][X_{M/F}] (-L_{1,F}-L_{b_1-1,E\setminus F})^{c_1-1} \\ \cdot S_{b_2-b_1,F}^{c_2} \cdots S_{b_k-b_1,F}^{c_k} L_{1,E \setminus F}^{r-\sum_{l=1}^k c_l} \end{multlined} \\
        &= \begin{multlined}[t] (-1)^{c_1-1}\sum_{F:|F|=n+1-b_1}\int_{X_{\Pi_F}}[X_{M|F}] S_{b_2-b_1,F}^{c_2} \cdots S_{b_k-b_1,F}^{c_k} L_{1,F}^{\rank_M(F)- 1 - \sum_{l=2}^k c_l} \\ \cdot\int_{X_{\Pi_{E\setminus F}}}[X_{M/F}] L_{1,E \setminus F}^{r-\sum_{l=1}^k c_l} L_{b_1-1,E \setminus F}^{\sum_{l=1}^k c_l - \rank_M(F)}
        \end{multlined} 
    \end{align}
    For these equations to be true, we define either integral to be 0 if any of the exponents are negative. In addition, if $|E \setminus F| = 1$, we define $L_{1,E \setminus F} = 0$. 
    Note that summands not appearing in \eqref{inductionstep} are 0 since they are push-forwards of classes of non-maximal degree.
    
    We shall compute the two smaller pushforwards individually. By induction, the first one is
    \begin{align}
    \int_{X_{\Pi_F}}& [X_{M|F}] S_{b_2-b_1,F}^{c_2} \cdots S_{b_k-b_1,F}^{c_k} L_{1,F}^{\rank_M(F)- 1 - \sum_{l=2}^k c_l} \nonumber \\
        &= (-1)^{\sum_{l=2}^kc_l-(k-1)}\smashoperator[l]{\sum_{\substack{\Ff\in \FlFl(k-1,M|F) \\ \forall i=1,\ldots,k-1 : \\ |F_i|=(n+1-b_1)-(b_{k+1-i}-b_1) }}} \prod_{i=1}^{k-1} \gamma_{((M | F)| F_{i+1})/F_{i}}\left( \rk_{M|F}(F_{i+1}) - 1 - \sum_{l=1}^i c_{k+1-l} \right) \nonumber \\
        &= (-1)^{\sum_{l=2}^kc_l-(k-1)}\smashoperator[l]{\sum_{\substack{\Ff\in \FlFl(k-1,M) \\ \forall i=1,\ldots,k-1 : \\ |F_i|=n+1-b_{k+1-i} \\ F_{k-1} \subsetneq F }}} \prod_{i=1}^{k-1} \gamma_{(M | F_{i+1})/F_{i}}\left( \rk_M(F_{i+1}) - 1 - \sum_{l=1}^i c_{k+1-l} \right) \label{oldpart}
    \end{align}
    
    The second pushforward is clearly zero whenever $M/F$ has a loop as then $[X_{M/F}]=0$. Therefore we only get a non-zero contribution for summands where $F$ is chosen such that $M/F$ is loopless. This condition is equivalent to $F$ being a flat in $M$. For a fixed flat $F$ the second pushforward is by definition
    \begin{align}
        \int_{X_{\Pi_{E\setminus F}}}[X_{M/F}] L_{1,E \setminus F}^{r-\sum_{l=1}^k c_l} L_{b_1-1,E \setminus F}^{\sum_{l=1}^k c_l - \rank_M(F)} &= \gamma_{M/F}\left(r - \sum_{l=1}^k c_l\right) \nonumber \\
        &= \gamma_{M/F}\left(\rank_M(E) - 1 - \sum_{l=1}^k c_{k+1-l}\right). \label{newpart}
    \end{align}
    Combining \eqref{inductionstep}, \eqref{oldpart}, and \eqref{newpart} gives the desired formula \eqref{intersectingS}. 
\end{proof}

We are now able to finish the proof of point (1) in Theorem \ref{thm:mixedEu} and thus of Corollary \ref{cor:mixedEuformula}.
\begin{thm}
    For $\Prod_{i=1}^k S_{b_i}^{c_i}\in \Mon^n(S_1,\dots,S_n)$ we have \[\int_{X_{\Perm_{n}}}\Prod_{i=1}^{k} S_{b_i}^{c_i}=(-1)^{n-k}\binom{n+1}{b_1,b_2-b_1,\dots,b_k-b_{k-1},n+1-b_k}\prod_{i=1}^{k-1}\binom{b_{i+1}-b_i-1}{\sum_{l=1}^ic_l-b_i}.\]
\end{thm}
    \begin{proof}
        By the proof of Theorem \ref{thm:matroidalmixed} for the uniform matroid $U_{n+1,n+1}$ of rank $n+1$ on $[n]$ we know that:
    \begin{align*}
    \int_{\Pi_n} \prod_{i=1}^kS_{b_i}^{c_i} &= (-1)^{n-k} \smashoperator[l]{\sum_{\substack{\Ff\in \FlFl(k,[n]) \\ \forall i=1,\ldots,k : \\ |F_i|=n+1-b_{k+1-i} }}} \prod_{i=1}^{k}\gamma_{U_{n+1,n+1}|{F_{i+1}/F_{i}}}\left(\rank_{U_{n+1,n+1}}(F_{i+1})-1-\sum_{l=1}^{i}c_{k+1-l}\right) \\
    &= (-1)^{n-k} \smashoperator[l]{\sum_{\substack{\Ff\in \FlFl(k,[n]) \\ \forall i=1,\ldots,k : \\ |F_i|=n+1-b_{k+1-i} }}} \prod_{i=1}^{k}\gamma_{U_{n+1,n+1}|{F_{i+1}/F_{i}}}\left(n - b_{k-i} -\sum_{l=1}^{i}c_{k+1-l}\right)
    \end{align*}
    Here we set $b_0=0$.
    Note that ${U_{n+1,n+1}|{F_{i+1}}/F_i}$ is $U_{b_{k+1-i}-b_{k-i},b_{k+1-i}-b_{k-i}}$. Thus all summands are equal and there are $\binom{n+1}{b_1,b_2-b_1,\dots,b_k-b_{k-1},n+1-b_k}$ many summands. It remains to note that 
    \begin{align*}
    \gamma_{U_{b_{k+1-i}-b_{k-i},b_{k+1-i}-b_{k-i}}} 
    \left(n - b_{k-i} -\sum_{l=1}^{i}c_{k+1-l}\right) &= \gamma_{U_{b_{k+1-i}-b_{k-i},b_{k+1-i}-b_{k-i}}} 
    \left(\sum_{l=1}^{k-i} c_l - b_{k-i}\right) \\
    &=\binom{b_{k+1-i}-b_{k-i}-1}{\sum_{l=1}^{k-i} c_l - b_{k-i}}.
    \end{align*}
    where we have used $\sum_{l=1}^k c_l = n$. Therefore, each summand is
    \begin{align*}
    \prod_{i=1}^k \binom{b_{k+1-i}-b_{k-i}-1}{\sum_{l=1}^{k-i} c_l - b_{k-i}} &=
    \prod_{i=0}^{k-1} \binom{b_{i+1}-b_{i}-1}{\sum_{l=1}^{i} c_l - b_{i}} \\
    &= \prod_{i=1}^{k-1} \binom{b_{i+1}-b_{i}-1}{\sum_{l=1}^{i} c_l - b_{i}}
    \end{align*}
    where the first equality comes from replacing $i$ with $k-i$ and the second from $\binom{b_1-1}{0} = 1$.
    \end{proof}
While the formulas derived in Theorem \ref{thm:matroidalmixed} and Corollary \ref{cor:mixedEuformula} give explicit descriptions of matroidal and classical mixed Eulerian numbers, they might not be optimal for computing these quantities in practice. Our next result therefore contains a recursive approach to computing matroidal mixed Eulerian numbers. Suppose we want to compute the intersection number $[X_M]L_1^{a_1}\cdots L_n^{a_n}$. The general strategy is as follows. \begin{enumerate}
    \item Pick one divisor $L_j$ that appears to a non-zero power $a_j\neq 0$ and express it as a rational linear combination of $S_1,\ldots,S_n$ via the relation (\ref{eqn:lineartransLiSi}) to write \begin{align*}
        [X_M]L_1^{a_1}\cdots L_n^{a_n}=\sum_{i=1}^n (B_{n,S\to L})_{i,j}[X_M]S_iL_1^{a_1}\cdots L_{j-1}^{a_{j-1}}L_j^{a_j-1}L_{j+1}^{a_{j+1}}\cdots L_n^{a_n}.
    \end{align*}
    \item Realize intersection with a divisor $S_j$ as restricting to a product of two smaller permutohedral varieties. 
    \item Compute the intersection numbers on the two smaller factors recursively.
\end{enumerate}
We already computed the restrictions of all appearing classes in Lemma \ref{restrictions of ctop} and Lemma \ref{restrictions of Li} respectively. We are therefore ready to prove the following theorem.
\begin{thm}
\label{matroidal-version}
    Let $M$ be a matroid on the groundset $E=[n]$, let $a_1,\ldots,a_n\in \N_0$ satisfy $\sum_{i=1}^na_i=\rank(M)-2$ and let $j\in \{1,\ldots,n\}$. The matroidal mixed Eulerian numbers satisfy the following recursive formula. \begin{align*}
        &A_M(a_1,\ldots,a_{j-1},a_j+1,a_{j+1},\ldots,a_n)\\&= \sum_{F\in Z(a_1,\ldots,a_n)}(B_{n,S\to L})_{j,n+1-|F|}A_{M/F}(a_1,\ldots,a_{n-|F|})A_{M|F}(a_{n+2-|F|},\ldots,a_n)
    \end{align*}
    where the sum ranges over all elements of the set \begin{align*}
        Z(a_1,\ldots,a_n)=\left\{F \in \Fl(M)\setminus \{\emptyset,E\}\middle| a_{n+1-|F|}=0, \sum_{i=1}^{n-|F|}a_i=\corank(F)-1\right\}.
    \end{align*}
\end{thm}
The condition for a subset $F\subseteq E$ to belong to $Z(a_1,\ldots, a_n)$ can be dropped in this summation. Indeed if $F\notin Z(a_1,\ldots,a_n)$ then the corresponding summand vanishes due to over-intersection in one of its factors or due to the contracted matroid $M/F$ having loops.
\begin{proof}
    We have \begin{align*}
    L_{j,E}=\sum_{i=1}^n(B_{n,S\to L})_{i,j}S_i= \sum_{i=1}^n(B_{n,S\to L})_{i,j}\sum_{F:|F|=n+1-i}x_F=\sum_{\emptyset \subsetneq F \subsetneq E}(B_{n,S\to L})_{n+1-|F|,j}x_F.
\end{align*}
    Therefore \begin{align*}
        &A_M(a_1,\ldots,a_{j-1},a_{j}+1,a_{j+1},\ldots,a_n)\\&= \int_{X_{\Pi_E}}[X_M]L_{j,E}\prod_{i=1}^nL_{i,E}^{a_i} \\&=\int_{X_{\Pi_E}}\sum_{\emptyset \subsetneq F \subsetneq E}(B_{n,S\to L})_{n+1-|F|,j}x_F\left([X_M]\prod_{i=1}^nL_{i,E}^{a_i}\right) \\
        &= \sum_{\emptyset \subsetneq F \subsetneq E}(B_{n,S\to L})_{n+1-|F|,j}\int_{X_{\Pi_E}}x_F\left([X_M]\prod_{i=1}^nL_{i,E}^{a_i}\right).
    \end{align*}
    Now we can apply Lemma \ref{pushforward-formulas} to compute all appearing pushforwards on smaller permutohedral varieties by restricting to $F$. Hence the expression above equals \begin{align*}
        \sum_{\emptyset \subsetneq F \subsetneq E}(B_{n,S\to L})_{n+1-|F|,j}\int_{X_{\Pi_F}\times X_{\Pi_{E\setminus F}}}[X_M]|_F\left(\prod_{i=1}^n(L_{i,E}|_{F})^{a_i}\right).
    \end{align*}
    In Lemma \ref{restrictions of Li} and Lemma \ref{restrictions of ctop} we computed all of the restrictions. Substituting this turns the above into \begin{align*}
    &\sum_{\emptyset \subsetneq F \subsetneq E}(B_{n,S\to L})_{n+1-|F|,j}\int_{X_{\Pi_F}\times X_{\Pi_{E\setminus F}}}[X_{M|F}]\otimes [X_{M/F}]\\&\cdot \left(\prod_{i=n+1-|F|+1}^{n}(L_{i-(n+1-|F|), F}\otimes 1)^{a_i}\right)(0\otimes 0)^{a_{n+1-|F|}}\left(\prod_{i=1}^{n+1-|F|-1}(1\otimes L_{i,E\setminus F})^{a_i}\right)
    \end{align*}
Using the second equality from Lemma \ref{pushforward-formulas} we can split the pushforward into its two factors: \begin{align*}
 &\sum_{\emptyset \subsetneq F \subsetneq E}(B_{n,S\to L})_{n+1-|F|,j}0^{a_{n+1-|F|}}\int_{X_{\Pi_F}}[X_{M|F}]\left(\prod_{i=1}^{|F|-1}L_{i, F}^{a_{i+n+1-|F|}}\right) \\ 
        &\cdot \int_{X_{\Pi_{E\setminus F}}}[X_{M/F}]\left(\prod_{i=1}^{n+1-|F|-1}L_{i,E\setminus F}^{a_i}\right) \\
\end{align*}
For the summand corresponding to the set $F$ to be non-zero, the following conditions must be met: \begin{itemize}
    \item $F$ is a flat in $M$ since otherwise $M/F$ has a loop and hence $[X_{M/F}]=0$.
    \item $a_{n+1-|F|}=0$ since otherwise the factor $0^{a_{n+1-|F|}}$ annihilates this term.
    \item $\nullity_M(F)+\sum_{i=1}^{|F|-1}a_{i+n+1-|F|}=|F|-1$ since otherwise the first pushforward vanishes.
    \item $\corank(M)-\nullity_M(F)+\sum_{i=1}^{n+1-|F|-1}a_i= n+1-|F|-1$ since otherwise the second pushforward vanishes.
\end{itemize}
However, since the sum of all $a_i$ is $\rank(M)-2$ the last two conditions are in fact equivalent. Therefore the set \begin{align*}
     Z(a_1,\ldots,a_n)=\left\{F \in \Fl(M)\setminus \{\emptyset,E\}\middle| a_{n+1-|F|}=0, \sum_{i=1}^{n-|F|}a_i=\corank(F)-1\right\}
\end{align*}
 describes those sets $F$ for which the corresponding summand has a chance of being nonzero. We therefore conclude \begin{align*}
     &A_M(a_1,\ldots,a_{j-1},a_{j}+1,a_{j+1},\ldots,a_n)\\&= 
     \sum_{F\in Z(a_1,\ldots,a_n)}(B_{n,S\to L})_{n+1-|F|,j}\int_{X_{\Pi_F}}[X_{M|_F}]\left(\prod_{i=1}^{|F|-1}L_{i, F}^{a_{i+n+1-|F|}}\right) \\ 
        &\cdot \int_{X_{\Pi_{E\setminus F}}}[X_{M/F}]\left(\prod_{i=1}^{n+1-|F|-1}L_{i,E\setminus F}^{a_i}\right) \\
        &=\sum_{F\in Z(a_1,\ldots,a_n)}(B_{n,S\to L})_{n+1-|F|,j} A_{M/F}(a_1,\ldots,a_{n-|F|})A_{M|_F}(a_{n+2-|F|},\ldots,a_n)
 \end{align*}
 as claimed.
\end{proof}

Inspired by \cite[Proposition 4.4, (3)]{LaksovLascouxThorup} we provide one more recursive relation among matroidal mixed Eulerian numbers. To compute the number $A_M(a_1,\ldots,a_n)$ this formula suggests the following recursive algorithm \begin{itemize}
    \item Find the maximal index $j$ such that the divisor $L_j$ appears in the product $\ctop L_1^{a_1}\cdots L_n^{a_n}$. In other words $j=\max(i\mid a_i\neq 0)$.
    \item If $j=1$ then $L_1^{a_1}\cdots L_n^{a_n}=L_1^{r}$ is the lexicographically maximal monomial in the $L_i$ divisor of degree $r$, this intersection number is always 1. Else use the relation $S_{j-1}=-L_{j-2}+2L_{j-1}-L_{j}$ to replace one copy of $L_j$ by a sum of three other terms.
    \item Compute the two terms corresponding to the summands $-L_{j-2}$ and $2L_{j-1}$ recursively as these are intersection numbers of monomials in $L_i$'s which are lexicographically strictly bigger than the one we started from.
    \item Realize intersection with $S_{j-1}$ as passing to a product of smaller permutohedral varieties, where now at most one factor will be non-zero and can be computed recursively.
\end{itemize}
We summarize this approach in the following statement.
\begin{thm}
    Let $M$ be a matroid of rank $r+1$ on $n+1$ elements $E$. Let $\bfa \in \N_0^n$ satisfy $\sum_{i=1}^n a_i=r$ and let $j=\max(i\mid a_i>0)$. We have $A_M(\bfa)=1$ if $j=1$ and \begin{align*}
        A_M(\bfa)&=-A_M(a_1,\ldots,a_{j-3},a_{j-2}+1,a_{j-1},a_{j}-1,0,\cdots,0)\\&+ 2A_M(a_1,\ldots,a_{j-2},a_{j-1}+1,a_{j}-1,0,\cdots,0)\\&-\sum_{\substack{F \in \Fl(M)\\|F|=n-j+2\\ \rank_M(F)=\sum_{i=j-1}^{n}a_i\\ a_{j-1}=0}}A_{M/F}(a_1,\ldots,a_{j-2})
    \end{align*}
         if $j >1$. The first summand is understood to be zero if $j=2$ and the last summand appears only when $a_{j-1}=0$.
\end{thm}
\begin{proof}
    Write $L_j=-L_{j-2}+2L_{j-1}-S_{j-1}$, where $L_0=0$, then \begin{align*}
        \int_{X_{\Pi_E}}\ctop L_1^{a_1}\cdots L_{j}^{a_j}&= -\int_{X_{\Pi_E}}\ctop L_1^{a_1}\cdots L_{j}^{a_j-1}\cdot L_{j-2}\\&+ 2\int_{X_{\Pi_E}} \ctop L_1^{a_1}\cdots L_{j}^{a_j-1}L_{j-1}\\&-\int_{X_{\Pi_E}}\ctop L_1^{a_1}\cdots L_{j-1}^{a_{j-1}}S_{j-1}L_{j}^{a_j-1}
    \end{align*}
    The first two summands are again matroidal mixed Eulerian numbers, the last term is computed using restriction as in the previous proof as \begin{align*}
        \int_{X_{\Pi_E}}\ctop L_1^{a_1}\cdots L_{j-1}^{a_{j-1}}S_{j-1}L_{j}^{a_j-1} &= \sum_{|F|=n+1-(j-1)}\left(\int_{X_{\Pi_{E\setminus F}}}[X_{M/F}]L_1^{a_1}\cdots L_{j-2}^{a_{j-2}}\right)\\&\cdot 0^{a_{j-1}}\cdot \left(\int_{X_{\Pi_E}}[X_{M|F}]L_1^{a_j-1}\right)
    \end{align*}
A summand here is non-zero only when $F$ is a flat (so that $M/F$ is loopless), $c_j-1=0$ and the sum of the first $j-2$ exponents is $\rank(M/F)-1$. In this case, the second integral is always one.
\end{proof}

\subsection{Derksen's $\mathcal{G}$-invariant and the matrodal mixed Eulerian numbers}

We next pass to discuss Theorem \ref{thm:Ginv}.
This theorem  is strongly related to the articles \cite{hampe2017intersection, eur2023stellahedral, derksen2010valuative}. By the last two, one may identify the valuative group of loopless matroids with the homology group of the permutohedron by assigning to a class of a matroid $M$ the homology class $[X_M]$, see also Section \ref{subsec:ChowofPer}. Such classes do not provide \emph{invariants} of matroids, as isomorphic matroids give distinct classes. Thus, it is natural to project the class $[X_M]$ to the $\SS_{n+1}$ invariant part of the homology group $[X_{M,sym}]:=\sum_{g\in \SS_{n+1}} g[X_M]$. This procedure is similar to considering the isomorphism class of $M$, instead of $M$ itself. The projection, which one may also call symmetrisation, is a matroid invariant.  By  Lemma \ref{invariant part} representing this projection is exactly the same as providing all intersection numbers of $[X_M]$ with products of $L_i$'s. Indeed the product of the symmetrisation $[X_{M,sym}]$ with an arbitrary class $C$ equals the product of $[X_M]$  with the symmetrisation of $C$. In this way, we see that the symmetrized class is a valuative invariant. In the setting of \cite{derksen2010valuative} the module $P_M(d,r)$ may be identified with the $2r$-th homology group of $X_{\perm_{d-1}}$ and $P_M^{sym}(d,r)$ with the $\SS_d$ invariant part. This explains that the formula \cite[Theorem 1 (a)]{derksen2010valuative} coincides with the dimensions of the homology groups of Peterson variety, cf. \cite[Proposition 4.1]{abe2021geometry} and \cite{tymoczko2006linear}. \par
To get a grasp on the collection of all matroid invariants arising as combinations of matroidal mixed Eulerian numbers, we will change the perspective on the techniques we used so far. In the above, we were interested in intersections of the class $[X_M]$ with products of the divisors $L_1,\ldots,L_n$ and we used the transformation $B_{n,S\to L}$ to translate this problem into an intersection of the computationally easier divisors $S_1,\ldots,S_n$. Reversing this process, we are also able to write every intersection of $[X_M]$ with divisors $S_1,\ldots,S_n$ as a linear combination of matroidal mixed Eulerian numbers. In this way, we establish a linear combination of matroidal mixed Eulerian numbers which for every matroid $M$ evaluates to Derksen's $\mathcal{G}$-invariant of $M$. As this invariant is universal for all valuative matroid invariants, we deduce Theorem \ref{thm:Ginv}. \par
Let us start by recalling the definition of Derksen's $\mathcal{G}$-invariant. Our notation follows \cite{BoninKung_G}, for the original definition see \cite{derksen2009symmetric}.
For fixed $n,r\in \N_0$ a sequence $[\underline{r}]:=(r_0,\ldots,r_n)$ of $n+1$ numbers $r_i\in \{0,1\}$ with exactly $r+1$ ones and $n-r$ zeroes is called an \emph{$(n+1,r+1)$-sequence}. The $\mathcal{G}$-invariant of a rank $r+1$ matroid on $n+1$ elements is a formal $\ZZ$-linear combination of the $\binom{n+1}{r+1}$ symbols $[\underline{r}]$. Let $M$ be a matroid of rank $r+1$ on $[n]$. For any permutation $\sigma \in \mathfrak{S}_{n+1}$ we obtain an $(n+1,r+1)$-sequence $[\underline{r(\sigma)}]$ by setting $r_0=\rank_M(\sigma(0))$ and for $j=1,\ldots,n$ \begin{align*}
    r_j:=\rank_M(\{\sigma(0),\ldots,\sigma(j)\})-\rank_M(\{\sigma(0),\ldots,\sigma(j-1)\})
\end{align*}
 \begin{defi}   
    The \emph{$\mathcal{G}$-invariant of $M$ } is defined as \begin{align*}
        \mathcal{G}(M)=\sum_{\sigma \in \mathfrak{S}_{n+1}}[\underline{r(\sigma)}].
    \end{align*}
\end{defi}
The importance of the $\mathcal{G}$-invariant is founded in its close relation to a large class of matroid invariants, called \emph{valuative invariants}. To any matroid $M=(E,\mathcal{B})$ we associate its \emph{base polytope} \begin{align*}
    P_M:=\conv(e_B \mid B \in \mathcal{B}) \subseteq \R^E.
\end{align*}
The indicator function of $P_M$ is the function $\iota_{P_M}:\R^E\to \Z$ that sends a point $x$ to one if $x\in P_M$ and to zero otherwise. Let $\text{Val}(n+1,r+1)$ be the subgroup of the additive group $\Hom(\R^{n+1},\Z)$ spanned by all indicator functions of base polytopes of matroids of rank $r+1$ on $[n]$. \par
An assignment $M\mapsto v(M)\in G$ valued in some abelian group $G$ is called \emph{valuative} if there exists a group morphism $f:\text{Val}(n+1,r+1)\to G$ such that $v(M)=f(\iota_{P_M})$ for every matroid $M$ of rank $r+1$ on $n+1$ elements.
Many well-known matroid invariants turn out to be valuative, some of the most prominent examples include the Tutte-polynomial, the characteristic polynomial and also the $\mathcal{G}$-invariant. For a detailed study of valuativity we refer to \cite{ardilaRincon_valuative}. \par
The authors of \cite{derksen2010valuative} show that the $\mathcal{G}$-invariant is universal among all valuative invariants by proving the following theorem.
\begin{thm}[\cite{derksen2010valuative} Theorem 1.4]
\label{thm:universalInvariant}
    Let $v$ be a valuative matroid invariant with values in $G$, that is, an invariant such that for every $n,r$ the restriction of $v$ to matroids of rank $r+1$ on $n+1$ elements factors through a group morphism $\text{Val}(n+1,r+1)\to G$. Then there exists a specialization $\text{ev}$ of the variables $[r]$ to $G$ such that $v$ factors through the $\mathcal{G}$-invariant and through the induced morphism. In other words, for every matroid $M$ on $n+1$ elements and of rank $r+1$ we have \begin{align*}
        v(M)=\text{ev}(\mathcal{G}(M)).
    \end{align*}
\end{thm}
A large class of valuative invariants also arises in the context of matroid intersection numbers. Indeed, the assignment $M\mapsto [X_M]$ assigning to a matroid its class in the permutohedral Chow ring is valuative (see \cite[Proposition 5.6]{berget2023tautological}). Since the degree map is additive we deduce one implication of Theorem \ref{thm:Ginv} already.
\begin{cor}
\label{cor:easyImplicationG}
    Let $M_1,M_2$ be two matroids of rank $r+1$ on $n+1$ elements such that $\mathcal{G}(M_1)=\mathcal{G}(M_2)$. Then for any sequence $a_1,\ldots,a_n$ summing to $r$ we have \begin{align*}
        A_{M_1}(a_1,\ldots,a_n)=A_{M_2}(a_1,\ldots,a_n).
    \end{align*}
\end{cor}
For the other implication it suffices to write the $\mathcal{G}$-invariant as a linear combination of matroidal mixed Eulerian numbers. To do so we use a result by Bonin and Kung from \cite{BoninKung_G} which shows that the data contained in the $\mathcal{G}$-invariant is precisely the \emph{catenary data} of the matroid $M$. More precisely for any sequence $b_0,\ldots,b_{r+1}$ summing to $n+1$, Bonin and Kung construct a linear combination of $(n+1,r+1)$-sequences denoted $\gamma(b_0,\ldots,b_{r+1})$. They further show that the $\mathcal{G}$-invariant of any matroid $M$ of rank $r+1$ on $n+1$ elements can be written as a nonnegative $\ZZ$-linear combination of $\gamma(b_0,\ldots,b_{r+1})$ where the coefficients count chains in the lattice of flats of $M$. 
\begin{thm}[\cite{BoninKung_G},Theorem 3.3]
For a matroid $M$ of rank $r+1$ on $n+1$ elements and numbers $b_0,\ldots,b_{r+1}\in \N_0$ summing to $n+1$, denote by $\nu(M;b_0,\ldots,b_{r+1})$ the number of flags of flats $\mathcal{F}=(F_0,\ldots,F_{r+1})\in \FlFl(r,M)$ where $\rank_M(F_i)=i$ and $|F_0|=b_0, |F_i\setminus F_{i-1}|=b_i$ for $i=1,\ldots,r+1$. Then \begin{align}
\label{eq:GfromCatenary}
    \mathcal{G}(M)=\sum_{b}\nu(M;b_0,\ldots,b_{r+1})\gamma(b_0,\ldots,b_{r+1})
\end{align} 
where the sum ranges over all sequences $a=(b_0,\ldots,b_{r+1})$ of nonegative integers summing to $n+1$.
\end{thm}
This representation of the $\mathcal{G}$-invariant is advantageous for our purpose since we can obtain the catenary data $\nu(M;b_0,\ldots,b_{r+1})$ as intersection numbers. Indeed from the proof of Theorem \ref{thm:matroidalmixed} it follows that for a loopless matroid $M$ and a sequence $b_1,\ldots,b_{r+1}$ of positive integers summing to $n+1$ \begin{align*}
    \int_{X_{\Pi_E}}[X_M]S_{n+1-b_1}S_{n+1-b_1-b_2}\cdots S_{n+1-b_1-\cdots -b_r}=\nu(M;0,b_1,\ldots,b_r,n+1-b_1-\cdots-b_r) 
\end{align*}
Notice that since we only consider loopless matroids, the value of $\nu(M;b_0,\ldots,b_{r+1})$ can only be non-zero for $b_0=0$ and therefore the above equation describes all coefficients appearing in (\ref{eq:GfromCatenary}) as intersection numbers. We summarize our result in the following proposition. \begin{prop}
\label{prop:GfromMMEN}
    Let $n,r\in \N$. For a sequence $b=(b_1,\ldots,b_r)\in \N^r$ with $\sum_{i=1}^rb_i<n+1$. Using the linear transformation $A_{n,L\to S}$ we find coefficients $D^{c}(b)$ such that 
    \begin{align*}
        S_{n+1-b_1}S_{n+1-b_1-b_2}\cdots S_{n+1-b_1-\cdots -b_r}=\sum_{L^c\in \Mon^r(L_1,\ldots,L_n)}D^{c}(b)L^c
    \end{align*}
    Then for every loopless matroid of rank $r+1$ on $n+1$ elements \begin{align}
    \label{eq:GfromMMEN}
       \mathcal{G}(M)= \sum_{b} \left(\sum_{L^c\in \Mon^r(L_1,\ldots,L_n)}D^c(b)A_M(c)\right)\gamma\left(0,b_1,\ldots,b_r,n+1-\sum_{i=1}^rb_i\right)
    \end{align}
\end{prop}
In particular we obtain the missing implication of Theorem \ref{thm:Ginv}. 
\begin{thm}
    Let $M_1,M_2$ be two loopless matroids of rank $r+1$ on $n+1$ elements. Then $\mathcal{G}(M_1)=\mathcal{G}(M_2)$ if and only if for every sequence $a\in \N_0^n$ of $n$ integers summing to $r$ we have \begin{align*}
        A_{M_1}(a)=A_{M_2}(a).
    \end{align*}
\end{thm}
\begin{proof}
    We poved the first implication in Corollary \ref{cor:easyImplicationG}. Now assume that all matroidal mixed Eulerian numbers agree for $M_1,M_2$ then the right hand side of (\ref{eq:GfromMMEN}) agrees for both matroids and hence by Proposition \ref{prop:GfromMMEN} we conclude $\mathcal{G}(M_1)=\mathcal{G}(M_2)$. 
\end{proof}
Since the divisors $L_1,\ldots,L_n$ generate the $\mathfrak{S}_{n+1}$-invariant part of the permutohedral Chow ring, intersecting the class $[X_M]$ with all monomials in $L_1,\ldots,L_n$ contains the same data as intersecting the symmetrized matroid class $[X_{M,sym}]$ with all degree $r$ classes. Using this perspective, an equivalent formulation of the last theorem states the following.
\begin{thm}
    For loopless matroids the assignment $M \mapsto [X_{M,sym}]$ is a universal valuative matroid invariant, i.e. any valuative invariant $\varphi:M\mapsto \varphi(M)$ which vanishes on matroids with loops factors through $M \mapsto [X_{M,sym}]$.
\end{thm}
Proposition \ref{prop:GfromMMEN} allows us to write several classical matroid invariants in terms of matroidal mixed Eulerian numbers, assuming that we can write them in terms of the $\mathcal{G}$-invariant. One example is the Tutte polynomial, for which Derksen gave an explecit specialization of the formal variables $[r]$ under which $\mathcal{G}(M)$ becomes $T_M(x,y)$. For the Tutte polynomial we give another way to express $T_M(x,y)$ in terms of matroidal mixed Eulerian numbers by building on the formula for $T_M(1,y)$ derived in \cite{berget2023log}. For this we use the following formula for the Tutte-polynomial (see \cite[Proposition 5.12]{dupontFink_Tutte}).\begin{align}
\label{eq:Tutteformula}
    T_M(x,y)=\sum_{F\in \Fl(M)}(x-1)^{\rank(M)-\rank_M(F)}T_{M|F}(1,y)
\end{align}
\begin{prop}
\label{prop:Tutte-formula}
    Let $M$ be a rank $r+1$ matroid on $[n]$, then \begin{align*}
        T_M(x,y)&=\int_{X_{\Pi_E}}\ctop \sum_{i=1}^n\sum_{l=1}^r\sum_{d=0}^{i-1-r+l}\frac{1}{(r-l)!}S_{n+1-i}L_1^{l-1}\left(\prod_{k=1}^{r-l}L_{k+n+1-i+d}\right)(x-1)^ly^d \\
        &+(x-1)^{\rank(M)}+ \int_{\Pi_E}\ctop\left(\sum_{d=0}^{n-r}L_{1+d}\cdots L_{r+d}y^d\right)
    \end{align*}
\end{prop}
When rewriting $S$-divisors in terms of $L$-divisors this gives a formula of the Tutte polynomial only in terms of matroidal mixed Eulerian numbers. One reason why this result could be interesting is that apart from the class $\ctop$ it only uses the divisors $L_i,S_j$, both of which have natural generalization to the setting of complete collineation varieties and invariants of tensors.
\begin{proof}
First notice that the two terms in the second line of the formula in Proposition \ref{prop:Tutte-formula} correspond to the terms where $F=\emptyset$ and $F=E$ in (\ref{eq:Tutteformula}). The $F=E$ case is precisely \cite[Corollary 1.6]{berget2023log}. For the other terms we introduce a new sum indexed by $m$ in which only the summand $m=r-l$ can be non-zero. We can then rearrange the sums as follows \begin{align*}
        &\int_{X_{\Pi_E}}\ctop \sum_{i=1}^n\sum_{l=1}^r\sum_{m=0}^{i-1}\sum_{d=0}^{i-m-1}\frac{1}{m!}S_{n+1-i}L_1^{l-1}\left(\prod_{k=1}^mL_{k+n+1-i+d}\right)(x-1)^ly^d \\&=\sum_{i=1}^n\int_{X_{\Pi_E}}\ctop S_{n+1-i}\left(\sum_{l=1}^rL_1^{l-1}(x-1)^l\right)\left(\sum_{m=0}^{i-1}\frac{1}{m!}\sum_{d=0}^{i-m-1}y^d\prod_{k=1}^mL_{k+n+1-i+d}\right) \\&=
        \sum_{i=1}^n\sum_{\substack{F\in \Fl(M)\\|F|=i}}\left(\sum_{l=1}^r\int_{X_{\Pi_{E\setminus F}}}[X_{M/F}]L_1^{l-1}(x-1)^l\right)\\&\cdot\left(\sum_{m=0}^{i-1}\frac{1}{m!}\int_{X_{\Pi_{F}}}[X_{M|F}]\sum_{d=0}^{i-m-1}L_{1+d}\cdots L_{m+d}y^d\right)
    \end{align*}
    The degree map on $\Pi_{E/F}$ returns zero unless $l=\rank(M/F)=\rank(M)-\rank_M(F)$ for degree reasons. Hence the sum with index $l$ has only one non-zero summand which is $(x-1)^{\rank(M)-\rank_M(F)}$. Similarly the second degree map returns zero unless $m=\rank(M|F)$ and in that case we get precisely $T_{M|F}(1,y)$ due to the formula in \cite[Corollary 1.6]{berget2023log}. Hence in total we get precisely the summands of (\ref{eq:Tutteformula}) that correspond to proper non-empty flats which finishes the proof. 
\end{proof}

\bibliography{bibML}
\bibliographystyle{plain} 
\end{document}